\documentclass[fleqn]{article}

\usepackage[a4paper, left=2cm, right=2cm, top=3cm, bottom=3cm]{geometry}                           % page format
\usepackage[applemac]{inputenc}                                                                    % Mac OS X character set
\usepackage[T1]{fontenc}                                                                           % character set
\usepackage[colorlinks, linkcolor=black, citecolor=black, urlcolor=black]{hyperref}                % hyperlinked pdf-output
\usepackage{ifthen}                                                                                % if-then-package for writing macros
\usepackage{amsmath}                                                                               % AMS-LaTeX
\usepackage{amssymb}                                                                               % AMS-LaTeX
\usepackage{amsthm}                                                                                % AMS-LaTeX
\usepackage{upgreek}                                                                               % upright greek letters
\usepackage{tikz}                                                                                  % TikZ for graphics

\usetikzlibrary{matrix,calc}                                                  

\allowdisplaybreaks                                                                                % allow pagebreaks in math environments

\swapnumbers
\theoremstyle{definition}
\newtheorem{definition}{Definition}[section]
\newtheorem{corollary}[definition]{Corollary}
\newtheorem{example}[definition]{Example}
\newtheorem*{notation*}{Notation}
\newtheorem{proposition}[definition]{Proposition}
\newtheorem{remark}[definition]{Remark}
\newtheorem*{remark*}{Remark}
\newtheorem{theorem}[definition]{Theorem}
\newtheorem*{theorem*}{Theorem}

\makeatletter
\renewcommand*\p@enumii{}                                                                          % format iterated enumerate environments
\makeatother

\newcommand{\BIGOP}[1]{\mathop{\mathchoice%
{\raise-0.22em\hbox{\huge $#1$}}%
{\raise-0.05em\hbox{\Large $#1$}}%
{\hbox{\large $#1$}}%
{#1}}}

\def\morchoice#1#2{\begingroup\setbox0=\hbox{$#1\xrightarrow{#2}$}%
        \setbox1=\hbox{$#1\longrightarrow$}%
        \ifdim\wd0<\wd1
        \stackrel{#2}\longrightarrow
        \else
        \xrightarrow{#2}\fi\endgroup}

\newcommand{\booktitle}[1]{\textsl{#1}}                                                            % booknames
\newcommand{\eigenname}[1]{\textsc{#1}}                                                            % eigennames
\newcommand{\newnotion}[1]{\textit{#1}}                                                            % new notions

\newcommand{\nbd}{\nobreakdash-\hspace{0pt}}                                                       % no-break-dash

\DeclareMathOperator{\AutomorphismGroup}{\mathrm{Aut}}                                             % automorphism group
\DeclareMathOperator{\Cokernel}{\mathrm{Coker}}                                                    % cokernel of a morphism
\DeclareMathOperator{\GroupPart}{\mathrm{Gp}}	                                                      % group part of a crossed module
\DeclareMathOperator{\Image}{\mathrm{Im}}                                                          % image of a morphism
\DeclareMathOperator{\Kernel}{\mathrm{Ker}}                                                        % kernel of a morphism
\DeclareMathOperator{\ModulePart}{\mathrm{Mp}}                                                     % module part of a crossed module
\DeclareMathOperator{\Ob}{\mathrm{Ob}}                                                             % set of objects
\newcommand{\act}{\cdot}                                                                           % group action
\newcommand{\ascendinginterval}[2]{\lceil #1, #2 \rceil}                                           % ascending interval
\newcommand{\basis}{\mathrm{e}}                                                                    % basis of a free object
\newcommand{\bigcart}{\BIGOP{\times}}                                                              % cartesian product over a family
\newcommand{\canonicalepimorphism}[1][]{\uppi^{#1}}                                                % canonical epimorphism
\newcommand{\canonicalmonomorphism}[1][]{\upiota^{#1}}                                             % canonical monomorphism
\newcommand{\cart}{\times}                                                                         % cartesian product
\newcommand{\CoboundaryGroup}[1][]{\mathrm{B}^{#1}}                                                % coboundary group
\newcommand{\CocycleGroup}[1][]{\mathrm{Z}^{#1}}                                                   % cocycle group
\newcommand{\cocycleofextension}[1]{\mathrm{z}^{#1}}                                               % n-cocycle of ext. with resp. to lifting system
\newcommand{\cohomologyclassofextension}[1][]{\mathrm{cl}^{#1}}                                    % cohomology class of an extension
\newcommand{\CochainComplex}[1][]{\mathrm{Ch}^{#1}}                                                % cochain complex
\newcommand{\CohomologyGroup}[1][]{\mathrm{H}^{#1}}                                                % cohomology group
\newcommand{\comp}{\circ}                                                                          % composition
\newcommand{\CrMod}{\mathbf{CrMod}}                                                                % category of crossed modules
\newcommand{\descendinginterval}[2]{\lfloor #1, #2 \rfloor}                                        % descending interval
\newcommand{\differential}{\partial}                                                               % differential
\newcommand{\directprod}{\times}                                                                   % direct product
\newcommand{\extensionclass}{\mathrm{e}}                                                           % extension class of a 3-cohomology class
\newcommand{\ExtensionClasses}[2][]{\mathrm{Ext}_{#1}^{#2}}                                        % set of extension classes in a universe
\newcommand{\Extensions}[2][]{\underline{\mathrm{Ext}}_{#1}^{#2}}                                  % set of extensions in a Grothendieck universe
\newcommand{\extensionequivalent}[1][]{\approx_{#1}}                                               % extension equivalent
\newcommand{\Grp}{\mathbf{Grp}}                                                                    % category of groups
\newcommand{\HomotopyGroup}[1][]{\uppi_{#1}}                                                       % homotopy group
\newcommand{\id}{\mathrm{id}}                                                                      % identity
\newcommand{\inc}{\mathrm{inc}}                                                                    % inclusion morphism
\newcommand{\Integers}{\mathbb{Z}}                                                                 % set of integers
\newcommand{\intersection}{\cap}                                                                   % intersection
\newcommand{\isomorphic}{\cong}                                                                    % isomorphic
\newcommand{\kdelta}{\updelta}                                                                     % Kronecker delta
\newcommand{\Map}{\mathrm{Map}}                                                                    % mapping set
\newcommand{\map}{\rightarrow}                                                                     % map
\newcommand{\morphism}[1][]{\mathpalette\morchoice{#1}}                                            % morphism
\newcommand{\Naturals}{\mathbb{N}}                                                                 % set of natural numbers
\newcommand{\NormalSubgroupCrossedModule}[2]{[{#1} \normalsubgroupeq {#2}]}                        % normal subgroup crossed module
\newcommand{\normalsubgroupeq}{\trianglelefteq}                                                    % normal subgroup
\newcommand{\op}{\mathrm{op}}                                                                      % opposite category
\newcommand{\pointisation}[2][]{\ifthenelse{\equal{#1}{}}{{#2}^{\mathrm{pt}}}{{#2}^{\mathrm{pt}, #1}}} % pointisation of an n-cocycle
\newcommand{\pointiser}[2][]{\ifthenelse{\equal{#1}{}}{\mathrm{p}_{#2}}{\mathrm{p}_{#2}^{#1}}}     % pointiser of an n-cocycle
\newcommand{\quo}{\mathrm{quo}}                                                                    % quotient morphism
\newcommand{\Set}{\mathbf{Set}}                                                                    % category of sets
\newcommand{\StandardExtension}{\mathrm{E}}                                                        % standard extension with respect to a 3-cocycle
\newcommand{\standardsectionsystem}{\mathrm{s}}                                                    % standard section system
\newcommand{\structuremorphism}[1][]{\upmu^{#1}}                                                   % structure morphism of a crossed module
\newcommand{\TrivialHomomorphismCrossedModule}[2]{[{#1} \; {#2}]}                                  % trivial homomorphism module
\newcommand{\triv}{\mathrm{triv}}                                                                  % trivial morphism
\newcommand{\union}{\cup}                                                                          % union

\tikzset{equality/.style={-, double}}
\tikzset{exists/.style={dotted}}
\tikzset{diagram/.style={matrix of math nodes, row sep=#1, column sep=#1, text height=1.6ex, text depth=0.45ex, inner sep=0pt, nodes={inner sep=0.333em}}, diagram/.default=2.5em}

\title{The third cohomology group \\ classifies crossed module extensions}
\author{Sebastian Thomas}
\date{September 29, 2010}

\begin{document}

\maketitle

\renewcommand{\thefootnote}{\fnsymbol{footnote}}
% 20J06: Cohomology of groups
% 18D35: Structured objects in a category (group objects, etc.)
\renewcommand{\thefootnote}{\arabic{footnote}}

\begin{abstract}
We give an elementary proof of the well-known fact that the third cohomology group \(\CohomologyGroup[3](G, M)\) of a group \(G\) with coefficients in an abelian \(G\)-module \(M\) is in bijection to the set \(\ExtensionClasses{2}(G, M)\) of equivalence classes of crossed module extensions of \(G\) with \(M\).
\end{abstract}

\section{Introduction} \label{sec:introduction}

This manuscript does not claim originality.

Perhaps the best-known result from group cohomology is the Schreier theorem, which gives an interpretation of the second cohomology group \(\CohomologyGroup[2](G, M)\) of a group \(G\) with coefficients in an abelian \(G\)-module \(M\). More precisely, it states that \(\CohomologyGroup[2](G, M)\) classifies group extensions of \(G\) with \(M\) in the sense that there is a bijection from \(\CohomologyGroup[2](G, M)\) to the set of extension classes \(\ExtensionClasses{1}(G, M)\) of group extensions of \(G\) with \(M\). By such a group extension we mean a short exact sequence of groups
\[M \morphism[\iota] E \morphism[\pi] G\]
for which the induced \(G\)-module structure on \(M\) coincides with the given one.

To give an interpretation of \(\CohomologyGroup[3](G, M)\), one has to consider crossed module extensions of \(G\) with \(M\) instead of group extensions. Roughly said, a crossed module extension of \(G\) with \(M\) is a four term exact sequence equipped with extra data such that the middle two terms form a crossed module.

The aim of this manuscript is to prove the following well-known theorem.
\begin{theorem*}[{cf.~\cite[th.~4.5]{holt:1979:an_interpretation_of_the_cohomology_groups_h_n_g_m},~\cite[p.~310]{huebschmann:1980:crossed_n-fold_extensions_of_groups_and_cohomology},~\cite[th.~9.4]{ratcliffe:1980:crossed_extensions}}]
Given a group \(G\) and an abelian \(G\)-module \(M\), we have
\[\ExtensionClasses{2}(G, M) \isomorphic \CohomologyGroup[3](G, M).\]
\end{theorem*}

A priori, the set of crossed module extension classes \(\ExtensionClasses{2}(G, M)\) is actually only a set, while the third cohomology group \(\CohomologyGroup[3](G, M)\) is an abelian group. So ``isomorphic'' in this theorem means that there is a bijection between \(\ExtensionClasses{2}(G, M)\) and \(\CohomologyGroup[3](G, M)\), which can of course be used to transport a group structure to \(\ExtensionClasses{2}(G, M)\). However, we will not pursue that possibility in this manuscript.

The proof presented here follows a sketch of \eigenname{Brown}~\cite[ch.~IV, sec.~5]{brown:1982:cohomology_of_groups}, while the techniques involved originally go back to \eigenname{Eilenberg} and \eigenname{Mac\,Lane}~\cite{eilenberg_maclane:1947:cohomology_theory_in_abstract_groups_I},~\cite{eilenberg_maclane:1947:cohomology_theory_in_abstract_groups_II_group_extensions_with_a_non-abelian_kernel},~\cite{maclane:1949:cohomology_theory_in_abstract_groups_III_operator_homomorphisms_of_kernels}. The cohomology class associated to a given crossed module extension class is constructed using certain lifts or sections in the underlying exact sequence of a representing crossed module extension. Conversely, to a given cohomology class we attach the extension class of a standard extension. Our proof allows to conclude that extensions in the same extension class are connected by at most two elementary steps (see corollary~\ref{cor:extension_equivalence_allows_chains_of_length_2}, cf.\ also~\cite[lem.~3.3]{holt:1979:an_interpretation_of_the_cohomology_groups_h_n_g_m}).

This manuscript sets the stage for~\cite{thomas:2009:on_the_second_cohomology_group_of_a_simplicial_group}, where we study the second cohomology group of a crossed module and, more generally, of a simplicial group. In that article, we will make explicit use of the constructions presented in this manuscript, in particular of the chosen sections and the \(3\)-cocycle constructed from them.

The result also appears in~\cite[th.~9.4]{ratcliffe:1980:crossed_extensions}. There is a more general result giving an interpretation of \(\CohomologyGroup[n + 1](G, M)\) in terms of extensions for all \(n \geq 1\). It has been independently proven by \eigenname{Holt}~\cite[th.~4.5]{holt:1979:an_interpretation_of_the_cohomology_groups_h_n_g_m} and \eigenname{Huebschmann}~\cite[p.~310]{huebschmann:1980:crossed_n-fold_extensions_of_groups_and_cohomology} and states that there is a bijection between \(\CohomologyGroup[n + 1](G, M)\) and the set of equivalence classes of so-called crossed \(n\)-fold extensions. To prove this theorem, \eigenname{Holt} uses universal delta functors, while \eigenname{Huebschmann} works with projective crossed resolutions. The author does not know whether there exists a proof of this more general result using lifts and sections in the spirit of \eigenname{Schreier} and of \eigenname{Eilenberg} and \eigenname{Mac\,Lane}. A summary of the development leading to this result can be found in the historical note of \eigenname{Mac\,Lane}~\cite{maclane:1979:historical_note}. A version in terms of \eigenname{Loday}s \(n\)-categorical groups can be found in~\cite[th.~4.2]{loday:1982:spaces_with_finitely_many_non-trivial_homotopy_groups}. Finally, an interpretation of \(\CohomologyGroup[4](G, M)\) using \eigenname{Conduch{\'e}}s \(2\)-crossed modules can be found in~\cite[th.~4.7]{conduche:1984:modules_croises_generalises_de_longeur_2}.

\paragraph{Outline} \label{par:outline}

We start in section~\ref{sec:preliminaries} with some preliminaries on groups and crossed modules. We show in section~\ref{sec:componentwise_pointed_cochains} that group cohomology can be expressed using componentwise pointed cochains. In section~\ref{sec:crossed_module_extensions_and_their_equivalence_classes}, we give the definitions of crossed module extensions and consider some examples. Thereafter, we show in section~\ref{sec:the_associated_cohomology_class} how a \(3\)-cohomology class can be associated to a given crossed module extension. Conversely, in section~\ref{sec:the_standard_extension} we construct a standard extension with respect to a given cocycle and show that both constructions are mutually inverse. This finally proves the classification theorem.

\subsection*{Conventions and notations}

\begin{itemize}
\item The composite of morphisms \(f\colon X \map Y\) and \(g\colon Y \map Z\) is usually denoted by \(f g\colon X \map Z\).
\item Given a complex of abelian groups \(A\) such that \(A^n = 0\) for \(n < 0\), we usually do not denote these zero objects.
\item We use the notations \(\Naturals = \{1, 2, 3, \dots\}\) and \(\Naturals_0 = \Naturals \cup \{0\}\).
\item Given a map \(f\colon X \rightarrow Y\) and subsets \(X' \subseteq X\), \(Y' \subseteq Y\) with \(X' f \subseteq Y'\), we write \(f|_{X'}^{Y'}\colon X' \rightarrow Y', x' \mapsto x' f\). Moreover, we abbreviate \(f|_{X'} := f|_{X'}^Y\) and \(f|^{Y'} := f|_X^{Y'}\).
\item Given integers \(a, b \in \Integers\), we write \([a, b] := \{z \in \Integers \mid a \leq z \leq b\}\) for the set of integers lying between \(a\) and \(b\). If we need to specify orientation, then we write \(\ascendinginterval{a}{b} := (z \in \Integers \mid a \leq z \leq b)\) for the \newnotion{ascending interval} and \(\descendinginterval{a}{b} = (z \in \Integers \mid a \geq z \geq b)\) for the \newnotion{descending interval}. Whereas we formally deal with tuples, we use the element notation, for example we write \(\prod_{i \in \ascendinginterval{1}{3}} g_i = g_1 g_2 g_3\) and \(\prod_{i \in \descendinginterval{3}{1}} g_i = g_3 g_2 g_1\) or \((g_i)_{i \in \descendinginterval{3}{1}} = (g_3, g_2, g_1)\) for group elements \(g_1\), \(g_2\), \(g_3\).
\item Given tuples \((x_i)_{i \in I}\) and \((x_j)_{j \in J}\) with disjoint index sets \(I\) and \(J\), we write \((x_i)_{i \in I} \union (x_j)_{j \in J}\) for their concatenation.
\item Given groups \(G\) and \(H\), we denote by \(\triv\colon G \map H\) the trivial group homomorphism \(g \mapsto 1\).
\item Given a group \(G\), a subgroup \(U\) of \(G\) and a quotient group \(Q\) of \(G\), we denote by \(\inc = \inc^U\colon U \map G\) the inclusion \(u \mapsto u\) and by \(\quo = \quo^Q\colon G \map Q\) the quotient morphism.
\item Given a group homomorphism \(\varphi\colon G \map H\), we denote its kernel by \(\Kernel \varphi\), its cokernel by \(\Cokernel \varphi\) and its image by \(\Image \varphi\).
\item The distinguished point in a pointed set \(X\) will be denoted by \(* = *^X\).
\item The Kronecker delta is defined by
\[\kdelta_{x, y} = \begin{cases} 1 & \text{for \(x = y\),} \\ 0 & \text{for \(x \neq y\),} \end{cases}\]
where \(x\) and \(y\) are elements of some set.
\end{itemize}

\paragraph{A remark on Grothendieck universes} \label{par:a_remark_on_grothendieck_universes}

To avoid set-theoretical difficulties, we work with Grothendieck universes~\cite[exp.~I, sec.~0]{artin_grothendieck_verdier:1972:sga_4_1} in this manuscript. In particular, every category has an object \emph{set} and a morphism \emph{set}.

We suppose given a Grothendieck universe \(\mathfrak{U}\). A \newnotion{set in \(\mathfrak{U}\)} (or \(\mathfrak{U}\)-set) is a set that is an element of \(\mathfrak{U}\), a \newnotion{map in \(\mathfrak{U}\)} (or \(\mathfrak{U}\)-map) is a map between \(\mathfrak{U}\)-sets. The \newnotion{category of \(\mathfrak{U}\)-sets} consisting of the set of \(\mathfrak{U}\)-sets, that is, of \(\mathfrak{U}\), as object set and the set of \(\mathfrak{U}\)-maps as morphism set will be denoted by \(\Set_{(\mathfrak{U})}\). A \newnotion{group in \(\mathfrak{U}\)} (or \(\mathfrak{U}\)-group) is a group whose underlying set is a \(\mathfrak{U}\)-set, a \newnotion{group homomorphism in \(\mathfrak{U}\)} (or \(\mathfrak{U}\)-group homomorphism) is a group homomorphism between \(\mathfrak{U}\)-groups. The \newnotion{category of \(\mathfrak{U}\)-groups} consisting of the set of \(\mathfrak{U}\)-groups as object set and the set of \(\mathfrak{U}\)-group homomorphisms as morphism set will be denoted by \(\Grp_{(\mathfrak{U})}\). Similarly for (abelian) \(G\)-modules, etc.

Because we do not want to overload our text with the usage of Grothendieck universes, we may suppress them in notation, provided we work with a single fixed Grothendieck universe. For example, instead of
\begin{remark*}
We suppose given a Grothendieck universe \(\mathfrak{U}\). The forgetful functor \(\Grp_{(\mathfrak{U})} \map \Set_{(\mathfrak{U})}\) is faithful.
\end{remark*}
we may just write
\begin{remark*}
The forgetful functor \(\Grp \map \Set\) is faithful.
\end{remark*}

Grothendieck universes will play a role when we consider extension classes of crossed module extensions, cf.\ section~\ref{sec:crossed_module_extensions_and_their_equivalence_classes}.

\section{Preliminaries} \label{sec:preliminaries}

\subsection*{Sections and lifts} \label{ssec:sections_and_lifts}

We suppose given a category \(\mathcal{C}\), objects \(X, Y, Z \in \Ob \mathcal{C}\) and morphisms \(f \in {_{\mathcal{C}}}(X, Y)\), \(g \in {_{\mathcal{C}}}(Z, Y)\). A \newnotion{section} of \(f\) is a morphism \(s \in {_{\mathcal{C}}}(Y, X)\) such that \(s f = 1_Y\). A \newnotion{lift} of \(g\) along \(f\) is a morphism \(l \in {_{\mathcal{C}}}(Z, X)\) such that \(g = l f\).

If \(f\) is a retraction, then the sections of \(f\) are exactly the lifts of \(1_Y\) along \(f\). Moreover, every section \(s\) of \(f\) defines a lift \(l\) of \(g\) along \(f\) by \(l := g s\).

\subsection*{Free groups} \label{ssec:free_groups}

We suppose given a set \(X\). Recall that a \newnotion{free group} on \(X\) consists of a group \(F\) together with a map \(e\colon X \map F\) such that for every group \(G\) and every map \(f\colon X \map G\) there exists a unique group homomorphism \(\varphi\colon F \map G\) with \(f = e \varphi\).
\[\begin{tikzpicture}[baseline=(m-2-1.base)]
  \matrix (m) [diagram]{
    F & G \\
    X & \\};
  \path[->, font=\scriptsize]
    (m-1-1) edge[exists] node[above] {\(\varphi\)} (m-1-2)
    (m-2-1) edge node[left] {\(e\)} (m-1-1)
            edge node[right] {\(f\)} (m-1-2);
\end{tikzpicture}\]
By abuse of notation, we often refer to the free group (consisting of the group \(F\) and the map \(e\)) as well as to its underlying group by \(F\). The map \(e\) is said to be the (\newnotion{ordered}) \newnotion{basis} of the free group \(F\). Given a free group \(F\) on \(X\) with basis \(e\), we write \(\basis = \basis^F := e\). The elements of \(\Image \basis^F\) are called \newnotion{free generators} of \(F\).

There exists a free group on every set \(X\), see for example~\cite[ch.~I, prop.~12.1]{lang:2002:algebra}. Moreover, the Nielsen-Schreier Theorem states that every subgroup \(U\) of a free group \(F\) is again a free group, and it describes explicitly a set of free generators of \(U\), see for example~\cite[\S 36, p.~36]{kurosh:1960:the_theory_of_groups}. We will apply this theorem in proposition~\ref{prop:standard_extension_of_a_3-cocycle}.

Since a group has a natural underlying pointed set with the neutral element as distinguished point, we can also define free groups on pointed sets: We suppose given a pointed set \(X\). A \newnotion{free group} on \(X\) consists of a group \(F\) together with a pointed map \(e\colon X \map F\) such that for every group \(G\) and every pointed map \(f\colon X \map G\) there exists a unique group homomorphism \(\varphi\colon F \map G\) with \(f = e \varphi\).
\[\begin{tikzpicture}[baseline=(m-2-1.base)]
  \matrix (m) [diagram]{
    F & G \\
    X & \\};
  \path[->, font=\scriptsize]
    (m-1-1) edge[exists] node[above] {\(\varphi\)} (m-1-2)
    (m-2-1) edge node[left] {\(e\)} (m-1-1)
            edge node[right] {\(f\)} (m-1-2);
\end{tikzpicture}\]
By abuse of notation, we often refer to the free group as well as to its underlying group by \(F\). The morphism \(e\) is said to be the (\newnotion{ordered}) \newnotion{basis} of the free group \(F\). Given a free group \(F\) on \(X\) with basis \(e\), we write \(\basis = \basis^F := e\). The elements of \(\Image \basis^F\) are called \newnotion{free generators} of \(F\).

We suppose given a pointed set \(X\). Roughly speaking, a free group on \(X\) is a free group on the set \(X \setminus \{*\}\). More precisely: Given a free group \(F\) on the pointed set \(X\), we obtain a free group \(F'\) on the set \(X \setminus \{*\}\) with underlying group \(F\) and basis \(\basis^{F'} = \basis^F|_{X \setminus \{*\}}\). Conversely, given a free group \(F'\) on the set \(X \setminus \{*\}\), we obtain a free group \(F\) on the \(X\) pointed set with underlying group \(F'\) and basis \(\basis^F\) defined by
\[x \basis^F = \begin{cases} x \basis^{F'} & \text{if \(x \in X \setminus \{*\}\),} \\ 1 & \text{if \(x = *\).} \end{cases}\]

\subsection*{Group actions} \label{ssec:group_actions}

We suppose given a category \(\mathcal{C}\) and a group \(G\). Recall that a (\newnotion{group}) \newnotion{action} of \(G\) on an object \(X \in \Ob \mathcal{C}\) is a group homomorphism \(\alpha\colon G^\op \map \AutomorphismGroup_{\mathcal{C}} X\).

A \newnotion{\(G\)-module} consists of a (not necessarily abelian) group \(M\) together with an action \(\alpha\) of \(G\) on \(M\), that is, a group homomorphism \(\alpha\colon G^\op \map \AutomorphismGroup_{\Grp} M\). By abuse of notation, we often refer to the module over \(G\) as well as to its underlying group by \(M\). The action \(\alpha\) is called the \newnotion{\(G\)-action} of the \(G\)-module \(M\). Given a \(G\)-module \(M\) with \(G\)-action \(\alpha\), we often write \({^g}m := m (g \alpha)\) for \(m \in M\), \(g \in G\). A \(G\)-module \(M\) is said to be \newnotion{abelian} if its underlying group is abelian. As usual, we often write \(M\) additively in this case, and we write \(g m = g \act m := m (g \alpha)\) for \(m \in M\), \(g \in G\), where \(\alpha\) denotes the \(G\)-action of \(M\).

A \(G\)-module structure on \(G\) itself is provided by the conjugation homomorphism \(G^\op \map \AutomorphismGroup G, g \mapsto {^g}(-)\), where \({^g}x = g x g^{- 1}\) for \(x, g \in G\).

\subsection*{Cohomology of groups} \label{ssec:cohomology_of_groups}

We suppose given an abelian \(G\)-module \(M\). The \newnotion{cochain complex} of \(G\) is the complex of abelian groups \(\CochainComplex(G, M) = \CochainComplex_{\Grp}(G, M)\) with entries \(\CochainComplex[n](G, M) := \Map(G^{\cart n}, M)\) and differentials given by
\begin{align*}
(g_j)_{j \in \descendinginterval{n}{0}} (c \differential) & = (g_{j + 1})_{j \in \descendinginterval{n - 1}{0}} c + \sum_{k \in [1, n]} (- 1)^k ((g_{j + 1})_{j \in \descendinginterval{n - 1}{k}} \union (g_k g_{k - 1}) \union (g_j)_{j \in \descendinginterval{k - 2}{0}}) c \\
& \qquad + (- 1)^{n + 1} g_n (g_j)_{j \in \descendinginterval{n - 1}{0}} c
\end{align*}
for \((g_j)_{j \in \descendinginterval{n}{0}} \in G^{\cart n}\), \(c \in \CochainComplex[n](G, M)\), \(n \in \Naturals_0\). Moreover, 
we define the \newnotion{\(n\)-th cocycle group} \(\CocycleGroup[n](G, M) := \CocycleGroup[n] \CochainComplex(G, M)\), the \newnotion{\(n\)-th coboundary group} \(\CoboundaryGroup[n](G, M) := \CoboundaryGroup[n] \CochainComplex(G, M)\) and the \newnotion{\(n\)-th cohomology group} \linebreak % manual format
\(\CohomologyGroup[n](G, M) := \CohomologyGroup[n] \CochainComplex(G, M) = \CocycleGroup[n](G, M) / \CoboundaryGroup[n](G, M)\) of \(G\) with coefficients in \(M\). An element \(c \in \CochainComplex[n](G, M)\) resp.\ \(z \in \CocycleGroup[n](G, M)\) resp.\ \(b \in \CocycleGroup[n](G, M)\) resp.\ \(h \in \CohomologyGroup[n](G, M)\) is said to be an \newnotion{\(n\)-cochain} resp.\ an \newnotion{\(n\)-cocycle} resp.\ an \newnotion{\(n\)-coboundary} resp.\ an \newnotion{\(n\)-cohomology class} of \(G\) with coefficients in \(M\).

\subsection*{Crossed modules} \label{ssec:crossed_modules}

A \newnotion{crossed module} consists of a group \(G\), a \(G\)-module \(M\) and a group homomorphism \(\mu\colon M \rightarrow G\) such that the following two axioms hold.
\begin{itemize}
\item[(Equi)] \emph{Equivariance}. We have \(({^g}m) \mu = {^g}(m \mu)\) for all \(m \in M\), \(g \in G\).
\item[(Peif)] \emph{Peiffer identity}. We have \({^{n \mu}}m = {^n}m\) for all \(m, n \in M\).
\end{itemize}
Here, \(G\) acts on \(G\) via conjugation, and so does \(M\) on \(M\). We call \(G\) the \newnotion{group part} and \(M\) the \newnotion{module part} of the crossed module. The group homomorphism \(\mu\colon M \rightarrow G\) is said to be the \newnotion{structure morphism} of the crossed module. Given a crossed module \(V\) with group part \(G\), module part \(M\) and structure morphism \(\mu\), we write \(\GroupPart V := G\), \(\ModulePart V := M\) and \(\structuremorphism = \structuremorphism[V] := \mu\).

We let \(V\) and \(W\) be crossed modules. A \newnotion{morphism of crossed modules} (or \newnotion{crossed module morphism}) from \(V\) to \(W\) consists of group homomorphisms \(\varphi_0\colon \GroupPart V \rightarrow \GroupPart W\) and \(\varphi_1\colon \ModulePart V \rightarrow \ModulePart W\) such that \(\varphi_1 \structuremorphism[W] = \structuremorphism[V] \varphi_0\) and such that \(({^g}m)\varphi_1 = {^{g \varphi_0}}(m \varphi_1)\) holds for all \(m \in \ModulePart V\), \(g \in \GroupPart V\). The group homomorphisms \(\varphi_0\) resp.\ \(\varphi_1\) are said to be the \newnotion{group part} resp.\ the \newnotion{module part} of the morphism of crossed modules. Given a crossed module morphism \(\varphi\) from \(V\) to \(W\) with group part \(\varphi_0\) and module part \(\varphi_1\), we write \(\GroupPart \varphi := \varphi_0\) and \(\ModulePart \varphi := \varphi_1\). Composition of morphisms of crossed modules is defined by composition on the group parts and on the module parts.

Let us consider two examples: Given a group \(G\) and a normal subgroup \(N \normalsubgroupeq G\), the inclusion \(\inc\colon N \map G\) together with the conjugation action of \(G\) on \(N\) yields the crossed module \(\NormalSubgroupCrossedModule{N}{G}\), called \newnotion{normal subgroup crossed module}. On the other hand, given a group \(G\) and an abelian \(G\)-module \(M\), the trivial homomorphism \(\triv\colon M \map G\) yields the crossed module \(\TrivialHomomorphismCrossedModule{M}{G}\), called \newnotion{trivial homomorphism crossed module}.

We let \(\mathfrak{U}\) be a Grothendieck universe. A crossed module \(V\) is said to be \newnotion{in \(\mathfrak{U}\)} (or a \newnotion{\(\mathfrak{U}\)-crossed module}) if \(\GroupPart V\) is a group in \(\mathfrak{U}\) and \(\ModulePart V\) is a \(G\)-module in \(\mathfrak{U}\). The \newnotion{category of \(\mathfrak{U}\)-crossed modules} consisting of \(\mathfrak{U}\)-crossed modules as objects and morphisms of \(\mathfrak{U}\)-crossed modules as morphisms will be denoted by \(\CrMod = \CrMod_{(\mathfrak{U})}\).

Given a crossed module \(V\), the image \(\Image \structuremorphism\) is a normal subgroup of \(\GroupPart V\) and the kernel \(\Kernel \structuremorphism\) is a central subgroup of \(\ModulePart V\). Moreover, the action of \(\GroupPart V\) on \(\ModulePart V\) restricts to a trivial action of \(\Image \structuremorphism\) on \(\Kernel \structuremorphism\). See for example~\cite[prop.~(5.3)]{thomas:2007:co_homology_of_crossed_modules}.

The \newnotion{homotopy groups} of \(V\) are defined by
\[\HomotopyGroup[n](V) := \begin{cases} \Cokernel \structuremorphism & \text{for \(n = 0\),} \\ \Kernel \structuremorphism & \text{for \(n = 1\),} \\ \{1\} & \text{for \(n \in \Naturals_0 \setminus \{0, 1\}\).} \end{cases}\]
Thus \(\HomotopyGroup[1](V)\) carries the structure of an abelian \(\HomotopyGroup[0](V)\)-module, where the action of \(\HomotopyGroup[0](V)\) on \(\HomotopyGroup[1](V)\) is induced by the action of \(\GroupPart V\) on \(\ModulePart V\), that is, for \(k \in \HomotopyGroup[1](V)\) and \(p \in \HomotopyGroup[0](V)\) we have \({^p}k = {^g}k\) for any \(g \in \GroupPart V\) with \(g (\Image \structuremorphism) = p\).

Given crossed modules \(V\) and \(W\), a crossed module morphism \(\varphi\colon V \map W\) is said to be a \newnotion{weak homotopy equivalence} if it induces isomorphisms \(\HomotopyGroup[n](V) \map \HomotopyGroup[n](W)\) for all \(n \in \Naturals_0\).

\begin{notation*}
Given a crossed module \(V\), the module part \(\ModulePart V\) resp.\ its opposite \((\ModulePart V)^{\op}\)) act on (the underlying set of) the group part \(\GroupPart V\) by \(m g := (m \structuremorphism) g\) and \(g m := g (m \structuremorphism)\) for \(m \in \ModulePart V\), \(g \in \GroupPart V\). Using this, we get for example
\[{^{m g}}n = {^{(m \structuremorphism) g}}n = {^{m \structuremorphism}}({^g}n) = {^m}({^g}n)\]
and
\[g m = g (m \structuremorphism) = {^g}(m \structuremorphism) g = (({^g} m) \structuremorphism) g = ({^g}m) g\]
for \(m, n \in \ModulePart V\), \(g \in \GroupPart V\). Also note that \((m g) n = m (g n)\) for \(m, n \in \ModulePart V\), \(g \in \GroupPart V\).

Given a set \(X\) and a map \(f\colon \GroupPart V \map X\), we usually write \(m f := m \structuremorphism f\) for \(m \in \ModulePart V\). Similarly for maps \(\GroupPart V \cart \GroupPart V \map X\), etc.

Moreover, given crossed modules \(V\) and \(W\) and a morphism of crossed modules \(\varphi\colon V \map W\), we may write \(m \varphi\) and \(g \varphi\) instead of \(m (\ModulePart \varphi)\) and \(g (\GroupPart \varphi)\). Using this, we have
\begin{align*}
(m g) \varphi & = ((m \structuremorphism[V]) g) (\GroupPart \varphi) = (m \structuremorphism[V] (\GroupPart \varphi)) (g (\GroupPart \varphi)) = (m (\ModulePart \varphi) \structuremorphism[W]) (g (\GroupPart \varphi)) = (m \varphi) (g \varphi)
\end{align*}
for \(m \in \ModulePart V\), \(g \in \GroupPart V\).
\end{notation*}

\section{Componentwise pointed cochains} \label{sec:componentwise_pointed_cochains}

We suppose given a group \(G\), an abelian \(G\)-module \(M\) and a non-negative integer \(n \in \Naturals_0\). Then \(G\) resp.\ \(M\) can naturally be considered as pointed sets with \(1\) resp.\ \(0\) as distinguished points. We want to make use of those cochains of \(G\) with coefficients in \(M\) that preserve these distinguished points.

This section follows~\cite[ch.~II, sec.~6]{eilenberg_maclane:1947:cohomology_theory_in_abstract_groups_I}.

\begin{definition}[componentwise pointed maps] \label{def:componentwise_pointed_maps}
We suppose given pointed sets \(X_i\) for \(i \in I\) and \(Y\), where \(I\) is an index set. A map \(f\colon \bigcart_{i \in I} X_i \map Y\) is said to be \newnotion{componentwise pointed} if \((x_i)_{i \in I} f = *\) for all \((x_i)_{i \in I} \in \bigcart_{i \in I} X_i\) with \(x_i = *\) for some \(i \in I\).
\end{definition}

\begin{definition}[componentwise pointed cochains] \label{def:componentwise_pointed_cochains}
The subset of \(\CochainComplex[n](G, M)\) consisting of all componentwise pointed \(n\)-cochains of \(G\) with coefficients in \(M\) will be denoted by
\[\CochainComplex[n]_{\text{cpt}}(G, M) = \CochainComplex[n]_{\Grp, \text{cpt}}(G, M) := \{c \in \CochainComplex[n](G, M) \mid \text{\(c\) componentwise pointed}\}.\]
Moreover, we set
\[\CocycleGroup[n]_{\text{cpt}}(G, M) = \CocycleGroup[n]_{\Grp, \text{cpt}}(G, M) := \CochainComplex[n]_{\text{cpt}}(G, M) \intersection \CocycleGroup[n](G, M)\]
for the set of componentwise pointed \(n\)-cocycles and
\[\CoboundaryGroup[n]_{\text{cpt}}(G, M) = \CoboundaryGroup[n]_{\Grp, \text{cpt}}(G, M) := \CochainComplex[n]_{\text{cpt}}(G, M) \intersection \CoboundaryGroup[n](G, M)\]
for the set of componentwise pointed \(n\)-coboundaries and
\[\CohomologyGroup[n]_{\text{cpt}}(G, M) = \CohomologyGroup[n]_{\Grp, \text{cpt}}(G, M) := \CocycleGroup[n]_{\text{cpt}}(G, M) / \CoboundaryGroup[n]_{\text{cpt}}(G, M)\]
for the set of componentwise pointed \(n\)-cohomology classes of \(G\) with coefficients in \(M\).
\end{definition}

The next definition is for technical purposes.

\begin{definition}[\(k\)-pointed cochains] \label{def:k-pointed_cochains}
We suppose given \(k \in [0, n]\). An \(n\)-cochain \(c \in \CochainComplex[n](G, M)\) is said to be \newnotion{\(k\)-pointed} if \((g_j)_{j \in \descendinginterval{n - 1}{0}} c = 0\) for all \((g_j)_{j \in \descendinginterval{n - 1}{0}} \in G^{\cart n}\) with \(g_j = 1\) for some \(j \in \descendinginterval{k - 1}{0}\).
\end{definition}

By definition, a \(0\)-pointed \(n\)-cochain is just an arbitrary \(n\)-cochain, while an \(n\)-pointed \(n\)-cochain is actually a componentwise pointed \(n\)-cochain.

\begin{remark} \label{rem:k-pointedness_and_differentials}
We suppose given \(k \in [0, n]\) and an \(n\)-cochain \(c \in \CochainComplex[n](G, M)\).
\begin{enumerate}
\item \label{rem:k-pointedness_and_differentials:k-pointedness} If \(c\) is \(k\)-pointed, then \(c \differential\) is \(k\)-pointed.
\item \label{rem:k-pointedness_and_differentials:componentwise_pointedness} If \(c\) is componentwise pointed, then \(c \differential\) is componentwise pointed.
\end{enumerate}
\end{remark}
\begin{proof} \
\begin{enumerate}
\item We suppose that \(c\) is \(k\)-pointed, and we let \(g_j \in G\) for \(j \in \descendinginterval{n}{0}\) be given with \(g_l = 1\) for some \(l \in \descendinginterval{k - 1}{0}\). For \(l = 0\), we have
\begin{align*}
(g_j)_{j \in \descendinginterval{n}{0}} (c \differential) & = (g_{j + 1})_{j \in \descendinginterval{n - 1}{0}} c + \sum_{i \in [1, n]} (- 1)^i ((g_{j + 1})_{j \in \descendinginterval{n - 1}{i}} \union (g_i g_{i - 1}) \union (g_j)_{j \in \descendinginterval{i - 2}{0}}) c \\
& \qquad + (- 1)^{n + 1} g_n (g_j)_{j \in \descendinginterval{n - 1}{0}} c \\
& = (g_{j + 1})_{j \in \descendinginterval{n - 1}{0}} c - ((g_{j + 1})_{j \in \descendinginterval{n - 1}{1}} \union (g_1)) c = 0,
\end{align*}
and for \(l \in \descendinginterval{k - 1}{1}\), we have
\begin{align*}
& (g_j)_{j \in \descendinginterval{n}{0}} (c \differential) \\
& = (g_{j + 1})_{j \in \descendinginterval{n - 1}{0}} c + \sum_{i \in [1, n]} (- 1)^i ((g_{j + 1})_{j \in \descendinginterval{n - 1}{i}} \union (g_i g_{i - 1}) \union (g_j)_{j \in \descendinginterval{i - 2}{0}}) c \\
& \qquad + (- 1)^{n + 1} g_n (g_j)_{j \in \descendinginterval{n - 1}{0}} c \\
& = (- 1)^l ((g_{j + 1})_{j \in \descendinginterval{n - 1}{l}} \union (g_{l - 1}) \union (g_j)_{j \in \descendinginterval{l - 2}{0}}) c \\
& \qquad + (- 1)^{l + 1} ((g_{j + 1})_{j \in \descendinginterval{n - 1}{l + 1}} \union (g_{l + 1}) \union (g_j)_{j \in \descendinginterval{l - 1}{0}}) c \\
& = (- 1)^l ((g_{j + 1})_{j \in \descendinginterval{n - 1}{l}} \union (g_j)_{j \in \descendinginterval{l - 1}{0}}) c + (- 1)^{l + 1} ((g_{j + 1})_{j \in \descendinginterval{n - 1}{l}} \union (g_j)_{j \in \descendinginterval{l - 1}{0}}) c = 0.
\end{align*}
Hence \(c \differential\) is also \(k\)-pointed.
\item We suppose that \(c\) is componentwise pointed. Then \(c \differential\) is \(n\)-pointed by~\ref{rem:k-pointedness_and_differentials:k-pointedness}. Moreover, given \(g_j \in G\) for \(j \in \descendinginterval{n - 1}{0}\), we have
\begin{align*}
& ((1) \union (g_j)_{j \in \descendinginterval{n - 1}{0}}) (c \differential) \\
& = ((1) \union (g_{j + 1})_{j \in \descendinginterval{n - 2}{0}}) c + \sum_{i \in [1, n - 1]} (- 1)^i ((1) \union (g_{j + 1})_{j \in \descendinginterval{n - 2}{i}} \union (g_i g_{i - 1}) \union (g_j)_{j \in \descendinginterval{i - 2}{0}}) c \\
& \qquad + (- 1)^n ((g_{n - 1}) \union (g_j)_{j \in \descendinginterval{n - 2}{0}}) c + (- 1)^{n + 1} (g_j)_{j \in \descendinginterval{n - 1}{0}} c \\
& = (- 1)^n (g_j)_{j \in \descendinginterval{n - 1}{0}} c + (- 1)^{n + 1} (g_j)_{j \in \descendinginterval{n - 1}{0}} c = 0. \qedhere
\end{align*}
\end{enumerate}
\end{proof}

\begin{definition}[componentwise pointisation of \(n\)-cocycles] \label{def:componentwise_pointisation_of_n-cocycles}
Given an \(n\)-cochain \(c \in \CochainComplex[n](G, M)\), the \newnotion{\(k\){\nbd}pointisation} \(\pointisation[k]{c} \in \CochainComplex[n](G, M)\) of \(c\) for \(k \in [0, n]\) is given recursively by
\[\pointisation[k]{c} := \begin{cases} c & \text{if \(k = 0\),} \\ \pointisation[k - 1]{c} - \pointiser[k]{c} \differential & \text{if \(k \in [1, n]\),} \end{cases}\]
where the \newnotion{\(k\)-pointiser} of \(c\) for \(k \in [1, n]\) is defined to be the \((n - 1)\)-cochain \(\pointiser[k]{c} \in \CochainComplex[n - 1](G, M)\) given by
\[(g_j)_{j \in \descendinginterval{n - 2}{0}} \pointiser[k]{c} := (- 1)^k ((g_{j - 1})_{j \in \descendinginterval{n - 1}{k}} \union (1) \union (g_j)_{j \in \descendinginterval{k - 2}{0}}) \pointisation[k - 1]{c}\]
for \(g_j \in G_j\), \(j \in \descendinginterval{n - 2}{0}\).
\end{definition}

\begin{proposition} \label{prop:k-pointisation_of_cochain_with_componentwise_pointed_coboundary_is_k-pointed}
We suppose given an \(n\)-cochain \(c \in \CochainComplex[n](G, M)\) such that \(c \differential\) is componentwise pointed. Then \(\pointisation[k]{c}\) is \(k\)-pointed for all \(k \in [0, n]\).
\end{proposition}
\begin{proof}
We proceed by induction on \(k \in [0, n]\), where for \(k = 0\) there is nothing to do. So let us suppose given \(k \in [1, n]\) and let us suppose that \(\pointisation[k - 1]{c}\) is \((k - 1)\)-pointed. Then \(\pointiser[k]{c}\) is \((k - 1)\)-pointed by definition and hence \(\pointisation[k]{c} = \pointisation[k - 1]{c} - \pointiser[k]{c} \differential\) is \((k - 1)\)-pointed by remark~\ref{rem:k-pointedness_and_differentials}\ref{rem:k-pointedness_and_differentials:k-pointedness}. It remains to show that
\[((g_j)_{j \in \descendinginterval{n - 1}{k}} \union (1) \union (g_j)_{j \in \descendinginterval{k - 2}{0}}) \pointisation[k]{c} = 0\]
for \(g_j \in G_j\), \(j \in \descendinginterval{n - 1}{k} \union \descendinginterval{k - 2}{0}\).
Indeed, if \(k \in [1, n - 1]\), then we have, since \(\pointisation[k - 1]{c}\) and \(\pointiser[k]{c}\) are \((k - 1)\)-pointed,
\begin{align*}
& ((g_j)_{j \in \descendinginterval{n - 1}{k}} \union (1) \union (g_j)_{j \in \descendinginterval{k - 2}{0}}) \pointisation[k]{c} \\
& = ((g_j)_{j \in \descendinginterval{n - 1}{k}} \union (1) \union (g_j)_{j \in \descendinginterval{k - 2}{0}}) \pointisation[k - 1]{c} - ((g_j)_{j \in \descendinginterval{n - 1}{k}} \union (1) \union (g_j)_{j \in \descendinginterval{k - 2}{0}}) (\pointiser[k]{c} \differential) \\
& = ((g_j)_{j \in \descendinginterval{n - 1}{k}} \union (1) \union (g_j)_{j \in \descendinginterval{k - 2}{0}}) \pointisation[k - 1]{c} \\
& \qquad - \sum_{i \in [k + 1, n - 1]} (- 1)^i ((g_{j + 1})_{j \in \descendinginterval{n - 2}{i}} \union (g_i g_{i - 1}) \union (g_j)_{j \in \descendinginterval{i - 2}{k}} \union (1) \union (g_j)_{j \in \descendinginterval{k - 2}{0}}) \pointiser[k]{c} \\
& \qquad - (- 1)^n g_{n - 1} ((g_j)_{j \in \descendinginterval{n - 2}{k}} \union (1) \union (g_j)_{j \in \descendinginterval{k - 2}{0}}) \pointiser[k]{c} \\
& = ((g_j)_{j \in \descendinginterval{n - 1}{k}} \union (1) \union (g_j)_{j \in \descendinginterval{k - 2}{0}}) \pointisation[k - 1]{c} \\
& \qquad - \sum_{i \in [k + 1, n - 1]} (- 1)^{i + k} ((g_j)_{j \in \descendinginterval{n - 1}{i + 1}} \union (g_i g_{i - 1}) \union (g_{j - 1})_{j \in \descendinginterval{i - 1}{k + 1}} \union (1) \union (1) \union (g_j)_{j \in \descendinginterval{k - 2}{0}}) \pointisation[k - 1]{c} \\
& \qquad - (- 1)^{n + k - 1} g_{n - 1} ((g_{j - 1})_{j \in \descendinginterval{n - 1}{k + 1}} \union (1) \union (1) \union (g_j)_{j \in \descendinginterval{k - 2}{0}}) \pointisation[k - 1]{c} \\
& = ((g_j)_{j \in \descendinginterval{n - 1}{k}} \union (1) \union (g_j)_{j \in \descendinginterval{k - 2}{0}}) \pointisation[k - 1]{c} \\
& \qquad + \sum_{i \in [k + 2, n]} (- 1)^{i + k - 1} ((g_j)_{j \in \descendinginterval{n - 1}{i}} \union (g_{i - 1} g_{i - 2}) \union (g_{j - 1})_{j \in \descendinginterval{i - 2}{k + 1}} \union (1) \union (1) \union (g_j)_{j \in \descendinginterval{k - 2}{0}}) \pointisation[k - 1]{c} \\
& \qquad + (- 1)^{n + k} g_{n - 1} ((g_{j - 1})_{j \in \descendinginterval{n - 1}{k + 1}} \union (1) \union (1) \union (g_j)_{j \in \descendinginterval{k - 2}{0}}) \pointisation[k - 1]{c} \\
& = (- 1)^{k - 1} \big( (- 1)^{k + 1} ((g_j)_{j \in \descendinginterval{n - 1}{k}} \union (1) \union (g_j)_{j \in \descendinginterval{k - 2}{0}}) \pointisation[k - 1]{c} \\
& \qquad + \sum_{i \in [k + 2, n]} (- 1)^i ((g_j)_{j \in \descendinginterval{n - 1}{i}} \union (g_{i - 1} g_{i - 2}) \union (g_{j - 1})_{j \in \descendinginterval{i - 2}{k + 1}} \union (1) \union (1) \union (g_j)_{j \in \descendinginterval{k - 2}{0}}) \pointisation[k - 1]{c} \\
& \qquad + (- 1)^{n + 1} g_{n - 1} ((g_{j - 1})_{j \in \descendinginterval{n - 1}{k + 1}} \union (1) \union (1) \union (g_j)_{j \in \descendinginterval{k - 2}{0}}) \pointisation[k - 1]{c} \big) \\
& = (- 1)^{k - 1} ((g_{j - 1})_{j \in \descendinginterval{n}{k + 1}} \union (1) \union (1) \union (g_j)_{j \in \descendinginterval{k - 2}{0}}) (\pointisation[k - 1]{c} \differential) \\
& = (- 1)^{k - 1} ((g_{j - 1})_{j \in \descendinginterval{n}{k + 1}} \union (1) \union (1) \union (g_j)_{j \in \descendinginterval{k - 2}{0}}) (c \differential) = 0.
\end{align*}
Moreover, for \(k = n\), we have \(n \geq 1\) and obtain, since \(\pointisation[n - 1]{c}\) and \(\pointiser[n]{c}\) are \((n - 1)\)-pointed,
\begin{align*}
((1) \union (g_j)_{j \in \descendinginterval{n - 2}{0}}) \pointisation[n]{c} & = ((1) \union (g_j)_{j \in \descendinginterval{n - 2}{0}}) \pointisation[n - 1]{c} - ((1) \union (g_j)_{j \in \descendinginterval{n - 2}{0}}) (\pointiser[n]{c} \differential) \\
& = ((1) \union (g_j)_{j \in \descendinginterval{n - 2}{0}}) \pointisation[n - 1]{c} = (- 1)^{n - 1} ((1) \union (1) \union (g_j)_{j \in \descendinginterval{n - 2}{0}}) (\pointisation[n - 1]{c} \differential) \\
& = (- 1)^{n - 1} ((1) \union (1) \union (g_j)_{j \in \descendinginterval{n - 2}{0}}) (c \differential) = 0. \qedhere
\end{align*}
\end{proof}

\pagebreak % manual format

\begin{corollary}[{cf.~\cite[lem.~6.1, lem.~6.2]{eilenberg_maclane:1947:cohomology_theory_in_abstract_groups_I}}] \label{cor:descriptions_of_componentwise_pointed_cocycle_coboundary_and_cohomology_group} \
\begin{enumerate}
\item \label{cor:descriptions_of_componentwise_pointed_cocycle_coboundary_and_cohomology_group:cocycles} We have
\[\CocycleGroup[n]_{\text{cpt}}(G, M) = \{z \in \CocycleGroup[n](G, M) \mid \text{\(\pointisation[k]{z} = \pointisation[k - 1]{z}\) for \(k \in [1, n]\)}\} = \{z \in \CocycleGroup[n](G, M) \mid \pointisation[n]{z} = z\}.\]
\item \label{cor:descriptions_of_componentwise_pointed_cocycle_coboundary_and_cohomology_group:coboundaries} If \(n \in \Naturals\), then we have
\[\CoboundaryGroup[n]_{\text{cpt}}(G, M) = (\CochainComplex[n - 1]_{\text{cpt}}(G, M)) \differential.\]
\item \label{cor:descriptions_of_componentwise_pointed_cocycle_coboundary_and_cohomology_group:cohomology} The embedding \(\CocycleGroup[n]_{\text{cpt}}(G, M) \map \CocycleGroup[n](G, M)\) and the \(n\)-pointisation homomorphism \(\CocycleGroup[n](G, M) \map \CocycleGroup[n]_{\text{cpt}}(G, M)\), % manual format
\(z \mapsto \pointisation[n]{z}\) induce mutually inverse isomorphisms between \(\CohomologyGroup[n]_{\text{cpt}}(G, M)\) and \(\CohomologyGroup[n](G, M)\). In particular,
\[\CohomologyGroup[n](G, M) \isomorphic \CohomologyGroup[n]_{\text{cpt}}(G, M).\]
\end{enumerate}
\end{corollary}
\begin{proof} \
\begin{enumerate}
\item We suppose given an \(n\)-cocycle \(z \in \CocycleGroup[n](G, M)\). If \(z\) is componentwise pointed, we inductively have \(\pointiser[k]{z} = 0\) and hence \(\pointisation[k]{z} = \pointisation[k - 1]{z}\) for \(k \in [1, n]\). If \(\pointisation[k]{z} = \pointisation[k - 1]{z}\) for \(k \in [1, n]\), it follows inductively that \(\pointisation[n]{z} = \pointisation[0]{z} = z\). Finally, if \(\pointisation[n]{z} = z\), it follows from the componentwise pointedness of \(\pointisation[n]{z} \differential = z \differential = 0\) that \(z = \pointisation[n]{z}\) is \(n\)-pointed by proposition~\ref{prop:k-pointisation_of_cochain_with_componentwise_pointed_coboundary_is_k-pointed}, that is, \(z\) is componentwise pointed.
\item By remark~\ref{rem:k-pointedness_and_differentials}\ref{rem:k-pointedness_and_differentials:componentwise_pointedness}, we have \((\CochainComplex[n - 1]_{\text{cpt}}(G, M)) \differential \subseteq \CoboundaryGroup[n]_{\text{cpt}}(G, M)\). Conversely, we suppose given an \(n\){\nbd}coboundary \(b \in \CoboundaryGroup[n](G, M)\) and we choose an \((n - 1)\)-cochain \(c \in \CochainComplex[n - 1](G, M)\) with \(b = c \differential\). Then we also have \(b = \pointisation[n - 1]{c} \differential\), and by proposition~\ref{prop:k-pointisation_of_cochain_with_componentwise_pointed_coboundary_is_k-pointed} it follows that if \(b\) is componentwise pointed, then so is \(\pointisation[n - 1]{c}\). Thus we also have \(\CoboundaryGroup[n]_{\text{cpt}}(G, M) = (\CochainComplex[n - 1]_{\text{cpt}}(G, M)) \differential\).

\item By definition of the \(n\)-pointisation, we have \(z = \pointisation[n]{z} + (\sum_{k \in [1, n]} \pointiser[k]{z}) \differential\) for every \(n\)-cocycle \(z \in \CocycleGroup[n](G, M)\) and since the \(n\)-pointisation \(\pointisation[n]{z}\) is componentwise pointed by~\ref{cor:descriptions_of_componentwise_pointed_cocycle_coboundary_and_cohomology_group:cocycles}, it follows that
\[\CohomologyGroup[n](G, M) = \CocycleGroup[n](G, M) / \CoboundaryGroup[n](G, M) = (\CocycleGroup[n]_{\text{cpt}}(G, M) + \CoboundaryGroup[n](G, M)) / \CoboundaryGroup[n](G, M).\]
Moreover,
\[\CohomologyGroup[n]_{\text{cpt}}(G, M) = \CocycleGroup[n]_{\text{cpt}}(G, M) / \CoboundaryGroup[n]_{\text{cpt}}(G, M) = \CocycleGroup[n]_{\text{cpt}}(G, M) / (\CocycleGroup[n]_{\text{cpt}}(G, M) \intersection \CoboundaryGroup[n](G, M)),\]
and thus Noether's first law of isomorphism provides the asserted isomorphisms
\begin{align*}
& \CohomologyGroup[n]_{\text{cpt}}(G, M) \map \CohomologyGroup[n](G, M), z + \CoboundaryGroup[n]_{\text{cpt}}(G, M) \mapsto z + \CoboundaryGroup[n](G, M) \text{ and} \\
& \CohomologyGroup[n](G, M) \map \CohomologyGroup[n]_{\text{cpt}}(G, M), z + \CoboundaryGroup[n](G, M) \mapsto \pointisation[n]{z} + \CoboundaryGroup[n]_{\text{cpt}}(G, M). \qedhere
\end{align*}
\end{enumerate}
\end{proof}

\section{Crossed module extensions and their equivalence classes} \label{sec:crossed_module_extensions_and_their_equivalence_classes}

In this section, we suppose given a group \(G\) and an abelian \(G\)-module \(M\).

\begin{definition}[crossed module extension] \label{def:crossed_module_extension} \
\begin{enumerate}
\item \label{def:crossed_module_extension:extensions} A \newnotion{crossed module extension} (or \newnotion{\(2\)-extension}) of \(G\) with \(M\) consists of a crossed module \(E\) together with a group monomorphism \(\iota\colon M \map \ModulePart E\) and a group epimorphism \(\pi\colon \GroupPart E \map G\) such that
\[M \morphism[\iota] \ModulePart E \morphism[\structuremorphism] \GroupPart E \morphism[\pi] G\]
is an exact sequence of groups and such that the induced action of \(G\) on \(M\) caused by the action of the crossed module \(E\) coincides with the a priori given action of \(G\) on \(M\), that is, such that \({^{e}}(m \iota) = ((e \pi) m) \iota\) for \(e \in \GroupPart E\) and \(m \in M\).

By abuse of notation, we often refer to the crossed module extension as well as to its underlying crossed module by \(E\). The morphism \(\iota\) is said to be the \newnotion{canonical monomorphism} and the morphism \(\pi\) is said to be the \newnotion{canonical epimorphism} of the crossed module extension \(E\).

Given a crossed module extension \(E\) of \(G\) with \(M\) with canonical monomorphism \(\iota\) and canonical epimorphism \(\pi\), we write \(\canonicalmonomorphism = \canonicalmonomorphism[E] := \iota\) and \(\canonicalepimorphism = \canonicalepimorphism[E] := \pi\).
\item \label{def:crossed_module_extension:set_of_extensions_in_a_grothendieck_universe} We suppose given a Grothendieck universe \(\mathfrak{U}\) such that \(G\) and \(M\) are in \(\mathfrak{U}\). A crossed module extension is said to be \newnotion{in \(\mathfrak{U}\)} (or a \newnotion{\(\mathfrak{U}\)-crossed module extension}) if its underlying crossed module is in \(\mathfrak{U}\). The set of crossed module extensions in \(\mathfrak{U}\) of \(G\) with \(M\) will be denoted by \(\Extensions{2}(G, M) = \Extensions[\mathfrak{U}]{2}(G, M)\).
\end{enumerate}
\end{definition}

\begin{remark} \label{rem:crossed_module_extensions_and_homotopy_groups_of_crossed_modules} \
\begin{enumerate}
\item \label{rem:crossed_module_extensions_and_homotopy_groups_of_crossed_modules:homotopy_groups_of_a_crossed_module_extension} We have \(\HomotopyGroup[0](E) \isomorphic G\) and \(\HomotopyGroup[1](E) \isomorphic M\) for every crossed module extension \(E\) of \(G\) with \(M\).
\item \label{rem:crossed_module_extensions_and_homotopy_groups_of_crossed_modules:canonical_extension_of_a_crossed_module} Conversely, given an arbitrary crossed module \(V\), we get a crossed module extension of \(\HomotopyGroup[0](V)\) with \(\HomotopyGroup[1](V)\), where \(\canonicalmonomorphism = \inc^{\HomotopyGroup[1](V)}\) and \(\canonicalepimorphism = \quo^{\HomotopyGroup[0](V)}\).
\[\HomotopyGroup[1](V) \morphism[\inc] \ModulePart V \morphism[\structuremorphism] \GroupPart V \morphism[\quo] \HomotopyGroup[0](V)\]
\end{enumerate}
\end{remark}

\begin{example} \label{ex:crossed_module_extensions} \
\begin{enumerate}
\item \label{ex:crossed_module_extensions:trivial_crossed_module_extension} The trivial homomorphism crossed module \(\TrivialHomomorphismCrossedModule{M}{G}\) provides a crossed module extension together with \(\id_M\) as canonical monomorphism and \(\id_G\) as canonical epimorphism, the \newnotion{trivial crossed module extension} of \(G\) with \(M\).
\[M \morphism[\id_M] M \morphism[\triv] G \morphism[\id_G] G\]
\item \label{ex:crossed_module_extensions:trivial_module} We suppose given a group \(E_0\) and a group epimorphism \(\pi\colon E_0 \map G\). Then the normal subgroup crossed module \(\NormalSubgroupCrossedModule{\Kernel \pi}{E_0}\) yields a crossed module extension of \(G\) with \(0\), where the canonical monomorphism is trivial and the canonical epimorphism is \(\pi\).
\[0 \morphism \Kernel \pi \morphism[\inc] E_0 \morphism[\pi] G\]
\end{enumerate}
\end{example}

\begin{definition}[equivalence of crossed module extensions] \label{def:equivalence_of_crossed_module_extensions} \
\begin{enumerate}
\item \label{def:equivalence_of_crossed_module_extensions:equivalence_morphisms} We let \(E\) and \(\tilde E\) be crossed module extensions of \(G\) with \(M\). An (\newnotion{extension}) \newnotion{equivalence} from \(E\) to \(\tilde E\) is a morphism of crossed modules \(\varphi\colon E \map \tilde E\) such that \(\canonicalmonomorphism[\tilde E] = \canonicalmonomorphism[E] (\ModulePart \varphi)\) and \(\canonicalepimorphism[E] = (\GroupPart \varphi) \canonicalepimorphism[\tilde E]\).
\[\begin{tikzpicture}[baseline=(m-2-1.base)]
  \matrix (m) [diagram]{
    M & \ModulePart E & \GroupPart E & G \\
    M & \ModulePart \tilde E & \GroupPart \tilde E & G \\};
  \path[->, font=\scriptsize]
    (m-1-1) edge node[above] {\(\canonicalmonomorphism[E]\)} (m-1-2)
            edge[equality] (m-2-1)
    (m-1-2) edge node[above] {\(\structuremorphism[E]\)} (m-1-3)
            edge node[right] {\(\ModulePart \varphi\)} (m-2-2)
    (m-1-3) edge node[above] {\(\canonicalepimorphism[E]\)} (m-1-4)
            edge node[right] {\(\GroupPart \varphi\)} (m-2-3)
    (m-1-4) edge[equality] (m-2-4)
    (m-2-1) edge node[above] {\(\canonicalmonomorphism[\tilde E]\)} (m-2-2)
    (m-2-2) edge node[above] {\(\structuremorphism[\tilde E]\)} (m-2-3)
    (m-2-3) edge node[above] {\(\canonicalepimorphism[\tilde E]\)} (m-2-4);
\end{tikzpicture}\]
\item \label{def:equivalence_of_crossed_module_extensions:equivalence_and_equivalence_classes} We suppose given a Grothendieck universe \(\mathfrak{U}\) such that \(G\) and \(M\) are in \(\mathfrak{U}\). We let \({\extensionequivalent} = {\extensionequivalent[\mathfrak{U}]}\) be the equivalence relation on \(\Extensions[\mathfrak{U}]{2}(G, M)\) generated by the following relation: Given extensions \(E, \tilde E \in \Extensions[\mathfrak{U}]{2}(G, M)\), the extension \(E\) is in relation to the extension \(\tilde E\) if there exists an extension equivalence \(E \map \tilde E\). Given crossed module extensions \(E\) and \(\tilde E\) with \(E \extensionequivalent \tilde E\), we say that \(E\) and \(\tilde E\) are (\newnotion{extension}) \newnotion{equivalent}. The set of equivalence classes of crossed module extensions in \(\mathfrak{U}\) of \(G\) with \(M\) with respect to \({\extensionequivalent[\mathfrak{U}]}\) is denoted by \(\ExtensionClasses{2}(G, M) = \ExtensionClasses[\mathfrak{U}]{2}(G, M) := \Extensions[\mathfrak{U}]{2}(G, M) / {\extensionequivalent[\mathfrak{U}]}\), and an element of \(\ExtensionClasses{2}(G, M)\) is said to be a \newnotion{crossed module extension class} of \(G\) with \(M\) \newnotion{in \(\mathfrak{U}\)} (or a \newnotion{\(\mathfrak{U}\)-crossed module extension class}).
\end{enumerate}
\end{definition}

\begin{remark} \label{rem:extension_equivalences_and_weak_equivalences} \
\begin{enumerate}
\item We suppose given a group \(G\) and an abelian \(G\)-module \(M\). Every extension equivalence \(\varphi\colon E \map \tilde E\) between crossed module extensions \(E\) and \(\tilde E\) of \(G\) with \(M\) is a weak homotopy equivalence between the underlying crossed modules of \(E\) and \(\tilde E\).
\item Given a weak homotopy equivalence \(\varphi\colon V \map W\) between crossed modules \(V\) and \(W\), there exist structures of crossed module extensions on \(V\) and \(W\) such that \(\varphi\) is an extension equivalence.
\end{enumerate}
\end{remark}

\begin{example} \label{ex:crossed_module_extension_equivalence_normal_subgroup_crossed_module}
We suppose given a group \(E_0\) and a group epimorphism \(\pi\colon E_0 \map G\). Then
\[\NormalSubgroupCrossedModule{\Kernel \pi}{E_0} \extensionequivalent \TrivialHomomorphismCrossedModule{0}{G}\]
since \(\pi\) induces an extension equivalence.
\[\begin{tikzpicture}[baseline=(m-2-1.base)]
  \matrix (m) [diagram]{
    0 & \Kernel \pi & E_0 & G \\
    0 & 0 & G & G \\};
  \path[->, font=\scriptsize]
    (m-1-1) edge (m-1-2)
            edge[equality] (m-2-1)
    (m-1-2) edge node[above] {\(\inc\)} (m-1-3)
            edge (m-2-2)
    (m-1-3) edge node[above] {\(\pi\)} (m-1-4)
            edge node[right] {\(\pi\)} (m-2-3)
    (m-1-4) edge[equality] (m-2-4)
    (m-2-1) edge[equality] (m-2-2)
    (m-2-2) edge (m-2-3)
    (m-2-3) edge[equality] (m-2-4);
\end{tikzpicture}\]
\end{example}

\section{The associated cohomology class} \label{sec:the_associated_cohomology_class}

During this section, we suppose given a group \(G\) and an abelian \(G\)-module \(M\).

The aim of this manuscript is to show that there is a bijection between the set of crossed module extension classes \(\ExtensionClasses{2}(G, M)\) and the third cohomology group \(\CohomologyGroup[3](G, M)\), see theorem~\ref{th:bijection_between_crossed_module_extension_classes_and_third_componentwise_pointed cohomology_group}. Since \(\CohomologyGroup[3](G, M) \isomorphic \CohomologyGroup[3]_{\text{cpt}}(G, M)\) by corollary~\ref{cor:descriptions_of_componentwise_pointed_cocycle_coboundary_and_cohomology_group}\ref{cor:descriptions_of_componentwise_pointed_cocycle_coboundary_and_cohomology_group:cohomology}, we are able to work with componentwise pointed cocycles and coboundaries. Most steps of the construction can be done with ``unpointed'' data, but it seems to the author that componentwise pointedness cannot be avoided in the proofs of proposition~\ref{prop:every_3-cocycle_of_the_cohomology_class_of_an_crossed_module_extension_is_constructable} and proposition~\ref{prop:every_cocycle_of_the_standard_extension_comes_from_a_section_system}. So for convenience, we will work with pointed sets and componentwise pointed maps throughout the whole procedure.

We start by constructing for a given crossed module extension class of \(G\) with \(M\) a cohomology class in \(\CohomologyGroup[3]_{\text{cpt}}(G, M)\). The arguments used here are adapted from~\cite[sec.~7]{eilenberg_maclane:1947:cohomology_theory_in_abstract_groups_II_group_extensions_with_a_non-abelian_kernel}.

\begin{remark} \label{rem:sections_of_group_epimorphisms_lead_to_non-abelian_2-cocycles}
We suppose given a group \(E_0\) and an epimorphism \(\pi\colon E_0 \map G\). For every section \(s^0\) of the underlying pointed map of \(\pi\), the map
\[\cocycleofextension{2} = \cocycleofextension{2}_{E_0, s^0}\colon G \cart G \map \Kernel \pi, (h, g) \mapsto (h s^0) (g s^0) ((h g) s^0)^{- 1}\]
is well-defined, componentwise pointed and fulfills
\[(k, h) \cocycleofextension{2} (k h, g) \cocycleofextension{2} = {^{k s^0}}((h, g) \cocycleofextension{2}) (k, h g) \cocycleofextension{2}\]
for \(g, h, k \in G\).
\end{remark}
\begin{proof}
We suppose given a section \(s^0\colon G \map E_0\) of the underlying pointed map of \(\pi\). Since \(\pi\) is a group homomorphism, we have \((h s^0) (g s^0) ((h g) s^0)^{- 1} \in \Kernel \pi\) for \(g, h \in G\). That is, we obtain a well-defined map \(\cocycleofextension{2}\colon G \cart G \map \Kernel \pi\) given by \((h, g) \cocycleofextension{2} := (h s^0) (g s^0) ((h g) s^0)^{- 1}\), that is, such that
\[(h s^0) (g s^0) = (h, g) \cocycleofextension{2} (h g) s^0\]
for \(g, h \in G\). Since \(s^0\) is pointed, we have
\begin{align*}
& (g, 1) \cocycleofextension{2} = (g s^0) (1 s^0) ((g 1) s^0)^{- 1} = 1 \text{ and} \\
& (1, g) \cocycleofextension{2} = (1 s^0) (g s^0) ((1 g) s^0)^{- 1} = 1
\end{align*}
for all \(g \in G\), that is, \(\cocycleofextension{2}\) is componentwise pointed. By computing the product \((k s^0) (h s^0) (g s^0)\) in \(E_0\) for \(g, h, k \in G\) in two different ways, we get on the one hand
\[((k s^0) (h s^0)) (g s^0) = (k, h) \cocycleofextension{2} (k h) s^0 (g s^0) = (k, h) \cocycleofextension{2} (k h, g) \cocycleofextension{2} (k h g) s^0,\]
and on the other hand
\[(k s^0) ((h s^0) (g s^0)) = (k s^0) (h, g) \cocycleofextension{2} (h g) s^0 = {^{k s^0}}((h, g)\cocycleofextension{2}) (k s^0) (h g) s^0 = {^{k s^0}}((h, g) \cocycleofextension{2}) (k, h g) \cocycleofextension{2} (k h g) s^0.\]
Hence \(\cocycleofextension{2}\) fulfills
\[((k, h) \cocycleofextension{2}) ((k h, g) \cocycleofextension{2}) = {^{k s^0}}((h, g) \cocycleofextension{2}) ((k, h g) \cocycleofextension{2})\]
for \(g, h, k \in G\).
\end{proof}

\begin{definition}[non-abelian \(2\)-cocycle of a crossed module extension] \label{def:non-abelian_2-cocycle_of_a_crossed_module_extension} \
\begin{enumerate}
\item We suppose given a group \(E_0\) and an epimorphism \(\pi\colon E_0 \map G\). Given a section \(s^0\) of the underlying pointed map of \(\pi\), we call
\[\cocycleofextension{2} = \cocycleofextension{2}_{E_0, s^0} \colon G \cart G \map \Kernel \pi, (h, g) \mapsto (h s^0) (g s^0) ((h g) s^0)^{- 1}\]
the \newnotion{non-abelian \(2\)-cocycle} of \(E_0\) with respect to \(s^0\). (\footnote{Note that \(\Kernel \pi\) is non-abelian in general. However, if \(\Kernel \pi\) is abelian, then \(\cocycleofextension{2}\) is the well-known \(2\)-cocycle in \(\CocycleGroup[2](G, \Kernel \pi)\) of the group extension \(E_0\) of \(G\) with \(\Kernel \pi\).})
\item Given a crossed module extension \(E\) of \(G\) with \(M\) and a section \(s^0\) of the underlying pointed map of \(\canonicalepimorphism\), the non-abelian \(2\)-cocycle \(\cocycleofextension{2}_{\GroupPart E, s^0}\) of \(\GroupPart E\) with respect to \(s^0\) is also said to be the \newnotion{non-abelian \(2\)-cocycle} of \(E\) with respect to \(s^0\) and is also denoted by \(\cocycleofextension{2} = \cocycleofextension{2}_{E, s^0} := \cocycleofextension{2}_{\GroupPart E, s^0}\).
\end{enumerate}
\end{definition}

\begin{definition}[lifting and section systems for crossed module extensions] \label{def:lifting_and_section_systems_for_crossed_module_extensions}
We suppose given a crossed module extension \(E\) of \(G\) with \(M\).
\begin{enumerate}
\item \label{def:lifting_and_section_systems_for_crossed_module_extensions:lifting_system} A \newnotion{lifting system} for \(E\) is a pair \((Z^2, Z^1)\) consisting of a lift \(Z^1\colon G \map \GroupPart E\) of \(\id_G\) along the underlying pointed map of \(\canonicalepimorphism\) and a lift \(Z^2\colon G \cart G \map \ModulePart E\) of \(\cocycleofextension{2}_{E, Z^1}\) along the underlying pointed map of \(\structuremorphism|^{\Image \structuremorphism}\) such that \(Z^2\) is componentwise pointed.
\item \label{def:lifting_and_section_systems_for_crossed_module_extensions:section_system} A \newnotion{section system} for \(E\) is a pair \((s^1, s^0)\) consisting of a section \(s^0\colon G \map \GroupPart E\) of the underlying pointed map of \(\canonicalepimorphism\) and a section \(s^1\colon \Image \structuremorphism \map \ModulePart E\) of the underlying pointed map of \(\structuremorphism|^{\Image \structuremorphism}\).
\end{enumerate}
\end{definition}

\begin{example} \label{ex:section_system_of_the_trivial_crossed_module_extension}
The unique section system for the trivial crossed module extension \(\TrivialHomomorphismCrossedModule{M}{G}\) of \(G\) with \(M\) is given by \((\triv, \id_G)\).
\end{example}

\begin{remark} \label{rem:section_systems_provide_lifting_systems}
We suppose given a crossed module extension \(E\) of \(G\) with \(M\). Every section system \((s^1, s^0)\) for \(E\) provides a lifting system \((Z^2, Z^1)\) for \(E\), where \(Z^1 := s^0\) and \(Z^2 := \cocycleofextension{2} s^1\).
\end{remark}
\begin{proof}
We suppose given a section system \((s^1, s^0)\) for \(E\). Then \(Z^1 := s^0\) is a section of \(\canonicalepimorphism\) and hence a lift of \(\id_G\) along the underlying pointed map of \(\canonicalepimorphism\). Further, \(Z^2 := \cocycleofextension{2} s^1\) is a lift of \(\cocycleofextension{2}\) along the underlying pointed map of \(\structuremorphism|^{\Image \structuremorphism}\). It is componentwise pointed since \(\cocycleofextension{2}\) is componentwise pointed by remark~\ref{rem:sections_of_group_epimorphisms_lead_to_non-abelian_2-cocycles} and \(s^1\) is pointed by assumption.
\end{proof}

\begin{definition}[lifting systems coming from section systems] \label{def:lifting_systems_coming_from_section_systems}
Given a crossed module extension \(E\) of \(G\) with \(M\) and a section system \((s^1, s^0)\) for \(E\), we say that a lifting system \((Z^2, Z^1)\) for \(E\) \newnotion{comes from} \((s^1, s^0)\) if \(Z^1 = s^0\) and \(Z^2 = \cocycleofextension{2} s^1\).
\end{definition}

\begin{remark} \label{rem:lifting_systems_of_crossed_module_extensions_lead_to_3-cocycles}
We suppose given a crossed module extension \(E\) of \(G\) with \(M\). For every lifting system \((Z^2, Z^1)\) for \(E\), the map
\begin{align*}
& \cocycleofextension{3} = \cocycleofextension{3}_{E, (Z^2, Z^1)}\colon G \cart G \cart G \map M, \\
& \qquad (k, h, g) \mapsto \big( (k, h) Z^2 (k h, g) Z^2 ((k, h g) Z^2)^{- 1} ({^{k Z^1}}((h, g) Z^2))^{- 1} \big) (\canonicalmonomorphism|^{\Image \canonicalmonomorphism})^{- 1}
\end{align*}
is a well-defined componentwise pointed \(3\)-cocycle of \(G\) with values in \(M\).
\end{remark}
\begin{proof}
We suppose given a lifting system \((Z^2, Z^1)\) for \(E\). By remark~\ref{rem:sections_of_group_epimorphisms_lead_to_non-abelian_2-cocycles}, we have
\[(k, h) \cocycleofextension{2} (k h, g) \cocycleofextension{2} = {^{k Z^1}}((h, g) \cocycleofextension{2}) (k, h g) \cocycleofextension{2}\]
for \(g, h, k \in G\). Hence it follows that \((k, h) Z^2 (k h, g) Z^2 ((k, h g) Z^2)^{- 1} ({^{k Z^1}}((h, g) Z^2))^{- 1} \in \Kernel \structuremorphism = \Image \canonicalmonomorphism\) for \(g, h, k \in G\). Since \(\canonicalmonomorphism\) is injective, we obtain a well-defined map \(\cocycleofextension{3}\colon G \cart G \cart G \map M\) given by \((k, h, g) \cocycleofextension{3} := \big( (k, h) Z^2 (k h, g) Z^2 ((k, h g) Z^2)^{- 1} ({^{k Z^1}}((h, g) Z^2))^{- 1} \big) (\canonicalmonomorphism|^{\Image \canonicalmonomorphism})^{- 1}\), that is, such that
\[(k, h) Z^2 (k h, g) Z^2 = (k, h, g) \cocycleofextension{3} \canonicalmonomorphism \, {^{k Z^1}}\!((h, g) Z^2) (k, h g) Z^2.\]
Since \(Z^1\) and \(Z^2\) are componentwise pointed, we have
\begin{align*}
& (h, g, 1) \cocycleofextension{3} = \big( (h, g) Z^2 (h g, 1) Z^2 ((h, g) Z^2)^{- 1} ({^{h Z^1}}((g, 1) Z^2))^{- 1} \big) (\canonicalmonomorphism|^{\Image \canonicalmonomorphism})^{- 1} = 0, \\
& (h, 1, g) \cocycleofextension{3} = \big( (h, 1) Z^2 (h, g) Z^2 ((h, g) Z^2)^{- 1} ({^{h Z^1}}((1, g) Z^2))^{- 1} \big) (\canonicalmonomorphism|^{\Image \canonicalmonomorphism})^{- 1} = 0, \\
& (1, h, g) \cocycleofextension{3} = \big( (1, h) Z^2 (h, g) Z^2 ((1, h g) Z^2)^{- 1} ({^{1 Z^1}}((h, g) Z^2))^{- 1} \big) (\canonicalmonomorphism|^{\Image \canonicalmonomorphism})^{- 1} = 0
\end{align*}
for all \(g, h \in G\), that is, \(\cocycleofextension{3}\) is also componentwise pointed. To show that \(\cocycleofextension{3} \in \CocycleGroup[3]_{\text{cpt}}(G, M)\), we compute \((l, k) Z^2 (l k, h) Z^2 (l k h, g) Z^2\) for \(g, h, k, l \in G\) in two different ways. On the one hand, we have
\begin{align*}
& (l, k) Z^2 (l k, h) Z^2 (l k h, g) Z^2 = (l, k) Z^2 (l k, h, g) \cocycleofextension{3} \canonicalmonomorphism \, {^{(l k) Z^1}}\!((h, g) Z^2) (l k, h g) Z^2 \\
& = (l k, h, g) \cocycleofextension{3} \canonicalmonomorphism (l, k) Z^2 \, {^{(l k) Z^1}}\!((h, g) Z^2) (l k, h g) Z^2 \\
& = (l k, h, g) \cocycleofextension{3} \canonicalmonomorphism \, {^{(l, k) Z^2 (l k) Z^1}}((h, g) Z^2) (l, k) Z^2 (l k, h g) Z^2 \\
& = (l k, h, g) \cocycleofextension{3} \canonicalmonomorphism \, {^{(l Z^1) (k Z^1)}}((h, g) Z^2) (l, k, h g) \cocycleofextension{3} \canonicalmonomorphism \, {^{l Z^1}}\!((k, h g) Z^2) (l, k h g) Z^2 \\
& = (l k, h, g) \cocycleofextension{3} \canonicalmonomorphism (l, k, h g) \cocycleofextension{3} \canonicalmonomorphism \, {^{(l Z^1) (k Z^1)}}((h, g) Z^2) \, {^{l Z^1}}\!((k, h g) Z^2) (l, k h g) Z^2 \\
& = ((l k, h, g) \cocycleofextension{3} + (l, k, h g) \cocycleofextension{3}) \canonicalmonomorphism \, {^{(l Z^1) (k Z^1)}}((h, g) Z^2) \, {^{l Z^1}}\!((k, h g) Z^2) (l, k h g) Z^2,
\end{align*}
and on the other hand, we get
\begin{align*}
& (l, k) Z^2 (l k, h) Z^2 (l k h, g) Z^2 = (l, k, h) \cocycleofextension{3} \canonicalmonomorphism \, {^{l Z^1}}\!((k, h) Z^2) (l, k h) Z^2 (l k h, g) Z^2 \\
& = (l, k, h) \cocycleofextension{3} \canonicalmonomorphism \, {^{l Z^1}}\!((k, h) Z^2) (l, k h, g) \cocycleofextension{3} \canonicalmonomorphism \, {^{l Z^1}}\!((k h, g) Z^2) (l, k h g) Z^2 \\
& = (l, k, h) \cocycleofextension{3} \canonicalmonomorphism (l, k h, g) \cocycleofextension{3} \canonicalmonomorphism \, {^{l Z^1}}\!((k, h) Z^2) \, {^{l Z^1}}\!((k h, g) Z^2) (l, k h g) Z^2 \\
& = (l, k, h) \cocycleofextension{3} \canonicalmonomorphism (l, k h, g) \cocycleofextension{3} \canonicalmonomorphism \, {^{l Z^1}}\!((k, h) Z^2 (k h, g) Z^2) (l, k h g) Z^2 \\
& = (l, k, h) \cocycleofextension{3} \canonicalmonomorphism (l, k h, g) \cocycleofextension{3} \canonicalmonomorphism \, {^{l Z^1}}\!((k, h, g) \cocycleofextension{3} \canonicalmonomorphism \, {^{k Z^1}}\!((h, g) Z^2) (k, h g) Z^2) (l, k h g) Z^2 \\
& = (l, k, h) \cocycleofextension{3} \canonicalmonomorphism (l, k h, g) \cocycleofextension{3} \canonicalmonomorphism \, {^{l Z^1}}\!((k, h, g) \cocycleofextension{3} \canonicalmonomorphism) \, {^{(l Z^1) (k Z^1)}}((h, g) Z^2) \, {^{l Z^1}}((k, h g) Z^2) (l, k h g) Z^2 \\
& = ((l, k, h) \cocycleofextension{3} + (l, k h, g) \cocycleofextension{3} + l \act (k, h, g) \cocycleofextension{3}) \canonicalmonomorphism \, {^{(l Z^1) (k Z^1)}}((h, g) Z^2) \, {^{l Z^1}}\!((k, h g) Z^2)) (l, k h g) Z^2.
\end{align*}
By the injectivity of \(\canonicalmonomorphism\), we conclude that
\[(l k, h, g) \cocycleofextension{3} + (l, k, h g) \cocycleofextension{3} = (l, k, h) \cocycleofextension{3} + (l, k h, g) \cocycleofextension{3} + l \act (k, h, g) \cocycleofextension{3}\]
for \(g, h, k, l \in G\), that is, \(\cocycleofextension{3} \in \CocycleGroup[3]_{\text{cpt}}(G, M)\).
\end{proof}

\begin{definition}[\(3\)-cocycle of a crossed module extension with respect to a lifting system] \label{def:3-cocycle_of_a_crossed_module_extension}
We suppose given a crossed module extension \(E\) of \(G\) with \(M\).
\begin{enumerate}
\item \label{def:3-cocycle_of_a_crossed_module_extension:lifting_system} Given a lifting system \((Z^2, Z^1)\) for \(E\), we call
\begin{align*}
& \cocycleofextension{3} = \cocycleofextension{3}_{E, (Z^2, Z^1)}\colon G \cart G \cart G \map M, \\
& \qquad (k, h, g) \mapsto \big( (k, h) Z^2 (k h, g) Z^2 ((k, h g) Z^2)^{- 1} ({^{k Z^1}}((h, g) Z^2))^{- 1} \big) (\canonicalmonomorphism|^{\Image \canonicalmonomorphism})^{- 1}
\end{align*}
the \newnotion{\(3\)-cocycle of \(E\)} with respect to \((Z^2, Z^1)\).
\item \label{def:3-cocycle_of_a_crossed_module_extension:section_system} Given a section system \((s^1, s^0)\), the \(3\)-cocycle of \(E\) with respect to the lifting system \((Z^2, Z^1)\) coming from \((s^1, s^0)\) is also called the \newnotion{\(3\)-cocycle of \(E\)} with respect to \((s^1, s^0)\) and denoted by \(\cocycleofextension{3} = \cocycleofextension{3}_{E, (s^1, s^0)} := \cocycleofextension{3}_{E, (Z^2, Z^1)}\).
\end{enumerate}
\end{definition}

\begin{example} \label{ex:3-cocycle_of_the_trivial_crossed_module_extension_with_respect_to_the_unique_section_system}
As we have seen in example~\ref{ex:section_system_of_the_trivial_crossed_module_extension}, the unique section system for \(\TrivialHomomorphismCrossedModule{M}{G}\) is given by \((\triv, \id_G)\). The \(3\)-cocycle of \(\TrivialHomomorphismCrossedModule{M}{G}\) with respect to \((\triv, \id_G)\) is the trivial \(3\)-cocycle \(0 \in \CocycleGroup[3]_{\text{cpt}}(G, M)\).
\end{example}

\begin{proposition} \label{prop:choices_of_liftings_of_the_non-abelian_2-cocycle_correspond_to_3-cocycles}
We suppose given a crossed module extension \(E\) of \(G\) with \(M\) and a lifting system \((Z^2, Z^1)\) for \(E\).
\begin{enumerate}
\item \label{prop:choices_of_liftings_of_the_non-abelian_2-cocycle_correspond_to_3-cocycles:correspondence_liftings_and_2-cochains} The maps \(\tilde Z^2\colon G \cart G \map \ModulePart E\) such that \((\tilde Z^2, Z^1)\) is a lifting system for \(E\) are exactly the maps of the form \(G \cart G \map \ModulePart E, (h, g) \mapsto (h, g) c^2 \canonicalmonomorphism (h, g) Z^2\) for some componentwise pointed \(2\)-cochain \(c^2 \in \CochainComplex[2]_{\text{cpt}}(G, M)\).
\item \label{prop:choices_of_liftings_of_the_non-abelian_2-cocycle_correspond_to_3-cocycles:3-cocycle_of_lifting} For every \(2\)-cochain \(c^2 \in \CochainComplex[2]_{\text{cpt}}(G, M)\), the \(3\)-cocycle \(\cocycleofextension{3}_{E, (\tilde Z^2, Z^1)}\) of \(E\) with respect to the lifting system \((\tilde Z^2, Z^1)\), where \((h, g) \tilde Z^2 := (h, g) c^2 \canonicalmonomorphism (h, g) Z^2\) for \(g, h \in G\), is given by \(\cocycleofextension{3}_{E, (\tilde Z^2, Z^1)} = c^2 \differential + \cocycleofextension{3}_{E, (Z^2, Z^1)}\).
\end{enumerate}
\end{proposition}
\begin{proof} \
\begin{enumerate}
\item First, we suppose given a componentwise pointed \(2\)-cochain \(c^2 \in \CochainComplex[2]_{\text{cpt}}(G, M)\). Then the map \(\tilde Z^2\colon G \cart G \map \ModulePart E, (h, g) \mapsto (h, g) c^2 \canonicalmonomorphism (h, g) Z^2\) is componentwise pointed since \(c^2\), \(Z^2\) and \(\canonicalmonomorphism\) are componentwise pointed. Moreover, we have \(\tilde Z^2 \structuremorphism|^{\Image \structuremorphism} = Z^2 \structuremorphism|^{\Image \structuremorphism} = \cocycleofextension{2}\), that is, \(\tilde Z^2\) is a lift of \(\cocycleofextension{2}\) along the underlying pointed map of \(\structuremorphism|^{\Image \structuremorphism}\). Hence \((\tilde Z^2, Z^1)\) is a lifting system for \(E\).

Conversely, we suppose given a lifting system \((\tilde Z^2, Z^1)\) for \(E\). Then \(Z^2\) and \(\tilde Z^2\) are componentwise pointed and lifts of \(\cocycleofextension{2}\) along the underlying pointed map of \(\structuremorphism|^{\Image \structuremorphism}\), that is, we have \((h, g) Z^2 \structuremorphism = (h, g) \tilde Z^2 \structuremorphism = (h, g) \cocycleofextension{2}\) for \(g, h \in G\). It follows that \((h, g) \tilde Z^2 ((h, g) Z^2)^{- 1} \in \Kernel \structuremorphism = \Image \canonicalmonomorphism\) for \(g, h \in G\). Hence we obtain a map \(c^2\colon G \cart G \map M, (h, g) \mapsto ((h, g) \tilde Z^2 ((h, g) Z^2)^{- 1}) (\canonicalmonomorphism|^{\Image \canonicalmonomorphism})^{- 1}\), that is, such that
\[(h, g) \tilde Z^2 = (h, g) c^2 \canonicalmonomorphism (h, g) Z^2\]
for \(g, h \in G\). Finally, \(c^2\) is componentwise pointed since \(Z^2\), \(\tilde Z^2\) and \(\canonicalmonomorphism\) are componentwise pointed.
\item We suppose given \(c^2 \in \CochainComplex[2]_{\text{cpt}}(G, M)\) and we define \(\tilde Z^2\colon G \cart G \map \ModulePart E, (h, g) \mapsto (h, g) c^2 \canonicalmonomorphism (h, g) Z^2\). By~\ref{prop:choices_of_liftings_of_the_non-abelian_2-cocycle_correspond_to_3-cocycles:correspondence_liftings_and_2-cochains}, \((\tilde Z^2, Z^1)\) is a lifting system for \(E\). We get
\begin{align*}
& (k, h) \tilde Z^2 (k h, g) \tilde Z^2 = (k, h) c^2 \canonicalmonomorphism (k, h) Z^2 (k h, g) c^2 \canonicalmonomorphism (k h, g) Z^2 = (k, h) c^2 \canonicalmonomorphism (k h, g) c^2 \canonicalmonomorphism (k, h) Z^2 (k h, g) Z^2 \\
& = (k, h) c^2 \canonicalmonomorphism (k h, g) c^2 \canonicalmonomorphism (k, h, g) \cocycleofextension{3}_{E, (Z^2, Z^1)} \canonicalmonomorphism \, {^{k Z^1}}\!((h, g) Z^2) (k, h g) Z^2 \\
& = (k, h) c^2 \canonicalmonomorphism (k h, g) c^2 \canonicalmonomorphism (k, h, g) \cocycleofextension{3}_{E, (Z^2, Z^1)} \canonicalmonomorphism \, {^{k Z^1}}\!(((h, g) c^2 \canonicalmonomorphism)^{- 1} (h, g) \tilde Z^2) ((k, h g) c^2 \canonicalmonomorphism)^{- 1} (k, h g) \tilde Z^2 \\
& = (k, h) c^2 \canonicalmonomorphism (k h, g) c^2 \canonicalmonomorphism (k, h, g) \cocycleofextension{3}_{E, (Z^2, Z^1)} \canonicalmonomorphism \, {^{k Z^1}}\!(((h, g) c^2 \canonicalmonomorphism)^{- 1}) \, {^{k Z^1}}\!((h, g) \tilde Z^2) ((k, h g) c^2 \canonicalmonomorphism)^{- 1} (k, h g) \tilde Z^2 \\
& = (k, h) c^2 \canonicalmonomorphism (k h, g) c^2 \canonicalmonomorphism ((k, h g) c^2 \canonicalmonomorphism)^{- 1} \, {^{k Z^1}}\!(((h, g) c^2 \canonicalmonomorphism)^{- 1}) (k, h, g) \cocycleofextension{3}_{E, (Z^2, Z^1)} \canonicalmonomorphism \, {^{k Z^1}}\!((h, g) \tilde Z^2) (k, h g) \tilde Z^2 \\
& = ((k, h) c^2 + (k h, g) c^2 - (k, h g) c^2 - k \act (h, g) c^2 + (k, h, g) \cocycleofextension{3}_{E, (Z^2, Z^1)}) \canonicalmonomorphism \, {^{k Z^1}}\!((h, g) \tilde Z^2) (k, h g) \tilde Z^2 \\
& = ((k, h, g) (c^2 \differential) + (k, h, g) \cocycleofextension{3}_{E, (Z^2, Z^1)}) \canonicalmonomorphism \, {^{k Z^1}}\!((h, g) \tilde Z^2) (k, h g) \tilde Z^2
\end{align*}
and thus \((k, h, g) \cocycleofextension{3}_{E, (\tilde Z^2, Z^1)} = (k, h, g) (c^2 \differential) + (k, h, g) \cocycleofextension{3}_{E, (Z^2, Z^1)}\) for \(g, h, k \in G\), that is,
\[\cocycleofextension{3}_{E, (\tilde Z^2, Z^1)} = c^2 \differential + \cocycleofextension{3}_{E, (Z^2, Z^1)}. \qedhere\]
\end{enumerate}
\end{proof}

\begin{proposition} \label{prop:cohomology_class_associated_to_a_crossed_module_extension}
We have a map
\[\cohomologyclassofextension\colon \Extensions{2}(G, M) \map \CohomologyGroup[3]_{\text{cpt}}(G, M)\]
that assigns to every crossed module extension \(E\) of \(G\) with \(M\) the cohomology class of the \(3\)-cocycle of \(E\) with respect to an arbitrarily chosen lifting system. The map \(\cohomologyclassofextension\) is independent from the chosen lifting system.
\end{proposition}
\begin{proof}
We suppose given a crossed module extension \(E\) of \(G\) with \(M\) and we choose a lifting system \((Z^2, Z^1)\) for \(E\). By proposition~\ref{prop:choices_of_liftings_of_the_non-abelian_2-cocycle_correspond_to_3-cocycles}, the cohomology class of \(\cocycleofextension{3}_{E, (Z^2, Z^1)}\) is independent from the choice of \(Z^2\). Thus it remains to show that the cohomology class of \(\cocycleofextension{3}_{E, (Z^2, Z^1)}\) is independent from the choice of \(Z^1\). To this end, we let \(\tilde Z^1\) be an alternative to \(Z^1\), that is, a section of the underlying pointed map of \(\canonicalepimorphism\). Then \((g \tilde Z^1) (g Z^1)^{- 1} \in \Kernel \canonicalepimorphism = \Image \structuremorphism\) for \(g \in G\), and we obtain a well-defined pointed map \(c^1\colon G \map \Image \structuremorphism, g \mapsto (g \tilde Z^1) (g Z^1)^{- 1}\), that is, such that
\[g \tilde Z^1 = (g c^1) (g Z^1)\]
for \(g \in G\). This implies
\begin{align*}
(h \tilde Z^1) (g \tilde Z^1) & = (h c^1) (h Z^1) (g c^1) (g Z^1) = (h c^1) \, {^{h Z^1}}\!(g c^1) (h Z^1) (g Z^1) = (h c^1) \, {^{h Z^1}}\!(g c^1) (h, g) \cocycleofextension{2}_{E, Z^1} (h g) Z^1 \\
& = (h c^1) \, {^{h Z^1}}\!(g c^1) (h, g) \cocycleofextension{2}_{E, Z^1} ((h g) c^1)^{- 1} (h g) \tilde Z^1
\end{align*}
and hence
\[(h, g) \cocycleofextension{2}_{E, \tilde Z^1} = (h c^1) \, {^{h Z^1}}\!(g c^1) (h, g) \cocycleofextension{2}_{E, Z^1} ((h g) c^1)^{- 1}\]
for \(g, h \in G\). We let \(C^1\colon G \map \ModulePart E\) be a lift of \(c^1\) along the underlying pointed map of \(\structuremorphism|^{\Image \structuremorphism}\), that is, a pointed map \(C^1\colon G \map \ModulePart E\) such that \(C^1 (\structuremorphism|^{\Image \structuremorphism}) = c^1\). Moreover, we define a lift \(\tilde Z^2\colon G \cart G \map \ModulePart E\) of \(\cocycleofextension{2}_{E, \tilde Z^1}\) along the underlying pointed map of \(\structuremorphism |^{\Image \structuremorphism}\) by \((h, g) \tilde Z^2 := (h C^1) \, {^{h Z^1}}\!(g C^1) (h, g) Z^2 ((h g) C^1)^{- 1}\) (\footnote{This is possible since the independence from the choice of this lift has already been shown.}), that is, such that
\[(h, g) \tilde Z^2 (h g) C^1 = (h C^1) \, {^{h Z^1}}\!(g C^1) (h, g) Z^2\]
for \(g, h \in G\). Since \(Z^1\), \(C^1\) and \(Z^2\) are componentwise pointed, we have
\begin{align*}
& (g, 1) \tilde Z^2 = (g C^1) \, {^{g Z^1}}\!(1 C^1) (g, 1) Z^2 (g C^1)^{- 1} = 1 \text{ and} \\
& (1, g) \tilde Z^2 = (1 C^1) \, {^{1 Z^1}}\!(g C^1) (1, g) Z^2 (g C^1)^{- 1} = 1
\end{align*}
for all \(g \in G\), whence \(\tilde Z^2\) is also componentwise pointed. Finally, we compute
\begin{align*}
(k, h) \tilde Z^2 (k h, g) \tilde Z^2 (k h g) C^1 & = (k, h) \tilde Z^2 (k h) C^1 \, {^{(k h) Z^1}}\!(g C^1) (k h, g) Z^2 \\
& = (k C^1) \, {^{k Z^1}}\!(h C^1) (k, h) Z^2 \, {^{(k h) Z^1}}\!(g C^1) (k h, g) Z^2 \\
& = (k C^1) \, {^{k Z^1}}\!(h C^1) \, {^{(k, h) Z^2 (k h) Z^1}}(g C^1) (k, h) Z^2 (k h, g) Z^2 \\
& = (k C^1) \, {^{k Z^1}}\!(h C^1) \, {^{(k Z^1) (h Z^1)}}(g C^1) (k, h, g) \cocycleofextension{3}_{E, (Z^2, Z^1)} \canonicalmonomorphism \, {^{k Z^1}}\!((h, g) Z^2) (k, h g) Z^2 \\
& = (k, h, g) \cocycleofextension{3}_{E, (Z^2, Z^1)} \canonicalmonomorphism (k C^1) \, {^{k Z^1}}\!(h C^1) \, {^{(k Z^1) (h Z^1)}}(g C^1) \, {^{k Z^1}}\!((h, g) Z^2) (k, h g) Z^2 \\
& = (k, h, g) \cocycleofextension{3}_{E, (Z^2, Z^1)} \canonicalmonomorphism (k C^1) \, {^{k Z^1}}\!((h C^1) \, {^{h Z^1}}\!(g C^1) (h, g) Z^2) (k, h g) Z^2 \\
& = (k, h, g) \cocycleofextension{3}_{E, (Z^2, Z^1)} \canonicalmonomorphism (k C^1) \, {^{k Z^1}}\!((h, g) \tilde Z^2 (h g) C^1) (k, h g) Z^2 \\
& = (k, h, g) \cocycleofextension{3}_{E, (Z^2, Z^1)} \canonicalmonomorphism \, {^{(k C^1) (k Z^1)}}((h, g) \tilde Z^2) (k C^1) \, {^{k Z^1}}\!((h g) C^1) (k, h g) Z^2 \\
& = (k, h, g) \cocycleofextension{3}_{E, (Z^2, Z^1)} \canonicalmonomorphism \, {^{(k c^1) (k Z^1)}}((h, g) \tilde Z^2) (k, h g) \tilde Z^2 (k h g) C^1 \\
& = (k, h, g) \cocycleofextension{3}_{E, (Z^2, Z^1)} \canonicalmonomorphism \, {^{k \tilde Z^1}}\!((h, g) \tilde Z^2) (k, h g) \tilde Z^2 (k h g) C^1
\end{align*}
for \(g, h, k \in G\). It follows that
\[(k, h) \tilde Z^2 (k h, g) \tilde Z^2 = (k, h, g) \cocycleofextension{3}_{E, (Z^2, Z^1)} \canonicalmonomorphism \, {^{k \tilde Z^1}}\!((h, g) \tilde Z^2) (k, h g) \tilde Z^2\]
for \(g, h, k \in G\), that is, \(\cocycleofextension{3}_{E, (\tilde Z^2, \tilde Z^1)} = \cocycleofextension{3}_{E, (Z^2, Z^1)}\).
\end{proof}

\begin{definition}[cohomology class associated to a crossed module extension] \label{def:cohomology_class_associated_to_a_crossed_module_extension} \
Given a crossed module extension \(E\) of \(G\) with \(M\), the cohomology class \(\cohomologyclassofextension(E) := \cocycleofextension{3}_{E, (Z^2, Z^1)} + \CoboundaryGroup[3]_{\text{cpt}}(G, M) \in \CohomologyGroup[3]_{\text{cpt}}(G, M)\) for an arbitrarily chosen lifting system \((Z^2, Z^1)\) for \(E\) is called the \newnotion{cohomology class associated to \(E\)}.
\end{definition}

\begin{example} \label{ex:cohomology_class_of_trivial_extension}
Following example~\ref{ex:3-cocycle_of_the_trivial_crossed_module_extension_with_respect_to_the_unique_section_system}, we have \(\cohomologyclassofextension(\TrivialHomomorphismCrossedModule{M}{G}) = 0\). 
\end{example}

Our next aim is to show that the cohomology class associated to a crossed module extension is independent from a chosen representative in its crossed module extension class.

\begin{proposition} \label{prop:compatible_choices_along_crossed_module_extensions_provide_compatible_cocycles}
We let \(E\) and \(\tilde E\) be crossed module extensions of \(G\) with \(M\) and we let \(\varphi\colon E \map \tilde E\) be an extension equivalence.
\[\begin{tikzpicture}[baseline=(m-2-1.base)]
  \matrix (m) [diagram]{
    M & \ModulePart E & \GroupPart E & G \\
    M & \ModulePart \tilde E & \GroupPart \tilde E & G \\};
  \path[->, font=\scriptsize]
    (m-1-1) edge node[above] {\(\canonicalmonomorphism[E]\)} (m-1-2)
            edge[equality] (m-2-1)
    (m-1-2) edge node[above] {\(\structuremorphism[E]\)} (m-1-3)
            edge node[right] {\(\ModulePart \varphi\)} (m-2-2)
    (m-1-3) edge node[above] {\(\canonicalepimorphism[E]\)} (m-1-4)
            edge node[right] {\(\GroupPart \varphi\)} (m-2-3)
    (m-1-4) edge[equality] (m-2-4)
    (m-2-1) edge node[above] {\(\canonicalmonomorphism[\tilde E]\)} (m-2-2)
    (m-2-2) edge node[above] {\(\structuremorphism[\tilde E]\)} (m-2-3)
    (m-2-3) edge node[above] {\(\canonicalepimorphism[\tilde E]\)} (m-2-4);
\end{tikzpicture}\]
\begin{enumerate}
\item \label{prop:compatible_choices_along_crossed_module_extensions_provide_compatible_cocycles:section_of_the_canonical_epimorphism} We suppose given a section \(s^0\) of the underlying pointed map of \(\canonicalepimorphism[E]\) and a section \(\tilde s^0\) of the underlying pointed map of \(\canonicalepimorphism[\tilde E]\) with \(\tilde s^0 = s^0 (\GroupPart \varphi)\). Then we have \(\cocycleofextension{2}_{\tilde E, \tilde s^0} = \cocycleofextension{2}_{E, s^0} (\GroupPart \varphi)|_{\Image \structuremorphism[E]}^{\Image \structuremorphism[\tilde E]}\).
\item \label{prop:compatible_choices_along_crossed_module_extensions_provide_compatible_cocycles:lifting_systems} We suppose given a lifting system \((Z^2, Z^1)\) for \(E\) and a lifting system \((\tilde Z^2, \tilde Z^1)\) for \(\tilde E\) with \(\tilde Z^1 = Z^1 (\GroupPart \varphi)\) and \(\tilde Z^2 = Z^2 (\ModulePart \varphi)\). Then we have \(\cocycleofextension{2}_{\tilde E, \tilde Z^1} = \cocycleofextension{2}_{E, Z^1} (\GroupPart \varphi)|_{\Image \structuremorphism[E]}^{\Image \structuremorphism[\tilde E]}\) and \(\cocycleofextension{3}_{\tilde E, (\tilde Z^2, \tilde Z^1)} = \cocycleofextension{3}_{E, (Z^2, Z^1)}\).

In particular, \(\cohomologyclassofextension(\tilde E) = \cohomologyclassofextension(E)\).
\item \label{prop:compatible_choices_along_crossed_module_extensions_provide_compatible_cocycles:section_systems} We suppose given a section system \((s^1, s^0)\) for \(E\) and a section system \((\tilde s^1, \tilde s^0)\) for \(\tilde E\) with \(\tilde s^0 = s^0 (\GroupPart \varphi)\) and \(s^1 (\ModulePart \varphi) = (\GroupPart \varphi)|_{\Image \structuremorphism[E]}^{\Image \structuremorphism[\tilde E]} \tilde s^1\). Moreover, we let \((Z^2, Z^1)\) be the lifting system coming from \((s^1, s^0)\) and \((\tilde Z^2, \tilde Z^1)\) be the section system coming from \((\tilde s^1, \tilde s^0)\). Then we have \(\tilde Z^1 = Z^1 (\GroupPart \varphi)\) and \(\tilde Z^2 = Z^2 (\ModulePart \varphi)\) as well as \(\cocycleofextension{2}_{\tilde E, \tilde s^0} = \cocycleofextension{2}_{E, s^0} (\GroupPart \varphi)|_{\Image \structuremorphism[E]}^{\Image \structuremorphism[\tilde E]}\) and \(\cocycleofextension{3}_{\tilde E, (\tilde s^1, \tilde s^0)} = \cocycleofextension{3}_{E, (s^1, s^0)}\).
\end{enumerate}
\end{proposition}
\begin{proof} \
\begin{enumerate}
\item We have
\begin{align*}
(h \tilde s^0) (g \tilde s^0) & = (h s^0 \varphi) (g s^0 \varphi) = ((h s^0) (g s^0)) \varphi = ((h, g) \cocycleofextension{2}_{E, s^0} (h g) s^0) \varphi = (h, g) \cocycleofextension{2}_{E, s^0} \varphi (h g) s^0 \varphi \\
& = (h, g) \cocycleofextension{2}_{E, s^0} \varphi (h g) \tilde s^0
\end{align*}
for \(g, h \in G\) and thus \(\cocycleofextension{2}_{\tilde E, \tilde s^0} = \cocycleofextension{2}_{E, s^0} (\GroupPart \varphi)|_{\Image \structuremorphism[E]}^{\Image \structuremorphism[\tilde E]}\).
\item By~\ref{prop:compatible_choices_along_crossed_module_extensions_provide_compatible_cocycles:section_of_the_canonical_epimorphism}, we have \(\cocycleofextension{2}_{\tilde E, \tilde Z^1} = \cocycleofextension{2}_{\tilde E, Z^1} (\GroupPart \varphi)|_{\Image \structuremorphism[E]}^{\Image \structuremorphism[\tilde E]}\). Further, we obtain
\begin{align*}
(k, h) \tilde Z^2 (k h, g) \tilde Z^2 & = (k, h) Z^2 \varphi (k h, g) Z^2 \varphi = ((k, h) Z^2 (k h, g) Z^2) \varphi \\
& = ((k, h, g) \cocycleofextension{3}_{E, (Z^2, Z^1)} \canonicalmonomorphism[E] \, {^{k Z^1}}\!((h, g) Z^2) (k, h g) Z^2) \varphi \\ 
& = (k, h, g) \cocycleofextension{3}_{E, (Z^2, Z^1)} \canonicalmonomorphism[E] \varphi \, {^{k Z^1 \varphi}}\!((h, g) Z^2 \varphi) (k, h g) Z^2 \varphi \\ 
& = (k, h, g) \cocycleofextension{3}_{E, (Z^2, Z^1)} \canonicalmonomorphism[\tilde E] \, {^{k \tilde Z^1}}\!((h, g) \tilde Z^2) (k, h g) \tilde Z^2 
\end{align*}
for \(g, h, k \in G\), that is, \(\cocycleofextension{3}_{\tilde E, (\tilde Z^2, \tilde Z^1)} = \cocycleofextension{3}_{E, (Z^2, Z^1)}\). In particular, it follows that
\[\cohomologyclassofextension(\tilde E) = \cocycleofextension{3}_{\tilde E, (\tilde Z^2, \tilde Z^1)} + \CoboundaryGroup[3]_{\text{cpt}}(G, M) = \cocycleofextension{3}_{E, (Z^2, Z^1)} + \CoboundaryGroup[3]_{\text{cpt}}(G, M) = \cohomologyclassofextension(E).\]
\item First of all, we get \(\cocycleofextension{2}_{\tilde E, \tilde s^0} = \cocycleofextension{2}_{E, s^0} (\GroupPart \varphi)|_{\Image \structuremorphism[E]}^{\Image \structuremorphism[\tilde E]}\) by~\ref{prop:compatible_choices_along_crossed_module_extensions_provide_compatible_cocycles:section_of_the_canonical_epimorphism}. Since the lifting systems \((Z^2, Z^1)\) resp.\ \((\tilde Z^2, \tilde Z^1)\) come from the section systems \((s^1, s^0)\) resp.\ \((\tilde s^1, \tilde s^0)\), we have \((Z^2, Z^1) = (\cocycleofextension{2}_{E, s^0} s^1, s^0)\) and \((\tilde Z^2, \tilde Z^1) = (\cocycleofextension{2}_{\tilde E, \tilde s^0} \tilde s^1, \tilde s^0)\). This implies
\[\tilde Z^1 = \tilde s^0 = s^0 (\GroupPart \varphi) = Z^1 (\GroupPart \varphi)\]
as well as
\[\tilde Z^2 = \cocycleofextension{2}_{\tilde E, \tilde s^0} \tilde s^1 = \cocycleofextension{2}_{E, s^0} (\GroupPart \varphi)|_{\Image \structuremorphism[E]}^{\Image \structuremorphism[\tilde E]} \tilde s^1 = \cocycleofextension{2}_{E, s^0} s^1 (\ModulePart \varphi) = Z^2 (\ModulePart \varphi).\]
Now~\ref{prop:compatible_choices_along_crossed_module_extensions_provide_compatible_cocycles:lifting_systems} yields
\[\cocycleofextension{3}_{\tilde E, (\tilde s^1, \tilde s^0)} = \cocycleofextension{3}_{\tilde E, (\tilde Z^2, \tilde Z^1)} = \cocycleofextension{3}_{E, (Z^2, Z^1)} = \cocycleofextension{3}_{E, (s^1, s^0)}. \qedhere\]
\end{enumerate}
\end{proof}

\begin{proposition} \label{prop:construction_of_compatible_choices_along_crossed_module_extensions}
We let \(E\) and \(\tilde E\) be crossed module extensions of \(G\) with \(M\) and we let \(\varphi\colon E \map \tilde E\) be an extension equivalence.
\begin{enumerate}
\item \label{prop:construction_of_compatible_choices_along_crossed_module_extensions:section_of_the_canonical_epimorphism} For every section \(s^0\) of the underlying pointed map of \(\canonicalepimorphism[E]\), the pointed map \(\tilde s_0 := s_0 (\GroupPart \varphi)\) is a section of the underlying pointed map of \(\canonicalepimorphism[\tilde E]\).
\item \label{prop:construction_of_compatible_choices_along_crossed_module_extensions:lifting_systems} We suppose given a lifting system \((Z^2, Z^1)\) for \(E\). Setting \(\tilde Z^1 := Z^1 (\GroupPart \varphi)\) and \(\tilde Z^2 := Z^2 (\ModulePart \varphi)\), we obtain a lifting system \((\tilde Z^2, \tilde Z^1)\) for \(\tilde E\).
\item \label{prop:construction_of_compatible_choices_along_crossed_module_extensions:sections_of_structure_morphism} For every section \(\tilde s^1\) of the underlying pointed map of \(\structuremorphism[\tilde E]|^{\Image \structuremorphism[\tilde E]}\), there exists exactly one section \(s^1\) of the underlying pointed map of \(\structuremorphism[E]|^{\Image \structuremorphism[E]}\) with \(s^1 (\ModulePart \varphi) = (\GroupPart \varphi)|_{\Image \structuremorphism[E]}^{\Image \structuremorphism[\tilde E]} \tilde s^1\). It is constructed as follows: For an arbitrarily chosen section \(s'^1\) of the underlying pointed map of \(\structuremorphism[E]|^{\Image \structuremorphism[E]}\), we have \(g s^1 = \big( (g \varphi \tilde s^1) (g s'^1 \varphi)^{- 1} \big) (\canonicalmonomorphism[\tilde E]|^{\Image \canonicalmonomorphism[\tilde E]})^{- 1} \canonicalmonomorphism[E] \, (g s'^1)\) for \(g \in \Image \structuremorphism[E]\).
\end{enumerate}
\end{proposition}
\begin{proof} \
\begin{enumerate}
\item We suppose given a section \(s^0\) of the underlying pointed map of \(\canonicalepimorphism[E]\). Setting \(\tilde s^0 := s^0 (\GroupPart \varphi)\), we get
\[\tilde s^0 \canonicalepimorphism[\tilde E] = s^0 (\GroupPart \varphi) \canonicalepimorphism[\tilde E] = s^0 \canonicalepimorphism[E] = \id_G,\]
that is, \(\tilde s^0\) is a section of the underlying pointed map of \(\canonicalepimorphism[\tilde E]\).
\item We have \(\tilde Z^1 \canonicalepimorphism[\tilde E] = \id_G\) by~\ref{prop:construction_of_compatible_choices_along_crossed_module_extensions:section_of_the_canonical_epimorphism} and therefore \(\cocycleofextension{2}_{\tilde E, \tilde Z^1} = \cocycleofextension{2}_{E, Z^1} (\GroupPart \varphi)|_{\Image \structuremorphism[E]}^{\Image \structuremorphism[\tilde E]}\) by  proposition~\ref{prop:compatible_choices_along_crossed_module_extensions_provide_compatible_cocycles}\ref{prop:compatible_choices_along_crossed_module_extensions_provide_compatible_cocycles:section_of_the_canonical_epimorphism}. Further, \(\tilde Z^2\) is a lift of \(\cocycleofextension{2}_{\tilde E, \tilde Z^1}\) along the underlying pointed map of \(\structuremorphism[\tilde E]|^{\Image \structuremorphism[\tilde E]}\) since
\[\tilde Z^2 (\structuremorphism[\tilde E]|^{\Image \structuremorphism[\tilde E]}) = Z^2 (\ModulePart \varphi) (\structuremorphism[\tilde E]|^{\Image \structuremorphism[\tilde E]}) = Z^2 (\structuremorphism[E]|^{\Image \structuremorphism[E]}) ((\GroupPart \varphi)|_{\Image \structuremorphism[E]}^{\Image \structuremorphism[\tilde E]}) = \cocycleofextension{2}_{E, Z^1} ((\GroupPart \varphi)|_{\Image \structuremorphism[E]}^{\Image \structuremorphism[\tilde E]}) = \cocycleofextension{2}_{\tilde E, \tilde Z^1},\]
and it is componentwise pointed since \(Z^2\) and \(\ModulePart \varphi\) are componentwise pointed. Thus \((\tilde Z^2, \tilde Z^1)\) is a lifting system for \(E\).
\item We suppose given a section \(\tilde s^1\) of the underlying pointed map of \(\structuremorphism[\tilde E]|^{\Image \structuremorphism[\tilde E]}\) and we choose a section \(s'^1\) of the underlying pointed map of \(\structuremorphism[E]|^{\Image \structuremorphism[E]}\). Then
\[g s'^1 \varphi \structuremorphism[\tilde E] = g s'^1 \structuremorphism[E] \varphi = g \varphi = g \varphi \tilde s^1 \structuremorphism[\tilde E]\]
and hence \((g \varphi \tilde s^1) (g s'^1 \varphi)^{- 1} \in \Kernel \structuremorphism[\tilde E] = \Image \canonicalmonomorphism[\tilde E]\) for \(g \in \Image \structuremorphism[E]\). Thus we obtain a well-defined pointed map
\[s^1\colon \Image \structuremorphism[E] \map \ModulePart E, g \mapsto \big( (g \varphi \tilde s^1) (g s'^1 \varphi)^{- 1} \big) (\canonicalmonomorphism[\tilde E]|^{\Image \canonicalmonomorphism[\tilde E]})^{- 1} \canonicalmonomorphism[E] \, (g s'^1)\]
with
\[g s^1 \structuremorphism[E] = \big( ((g \varphi \tilde s^1) (g s'^1 \varphi)^{- 1}) (\canonicalmonomorphism[\tilde E]|^{\Image \canonicalmonomorphism[\tilde E]})^{- 1} \canonicalmonomorphism[E] \, (g s'^1) \big) \structuremorphism[E] = g s'^1 \structuremorphism[E] = g\]
for all \(g \in \Image \structuremorphism[E]\), that is, \(s^1\) is a section of the underlying pointed map of \(\structuremorphism[E]|^{\Image \structuremorphism[E]}\). Moreover, we have
\begin{align*}
g s^1 \varphi & = \big( ((g \varphi \tilde s^1) (g s'^1 \varphi)^{- 1}) (\canonicalmonomorphism[\tilde E]|^{\Image \canonicalmonomorphism[\tilde E]})^{- 1} \canonicalmonomorphism[E] \, (g s'^1) \big) \varphi = ((g \varphi \tilde s^1) (g s'^1 \varphi)^{- 1}) (\canonicalmonomorphism[\tilde E]|^{\Image \canonicalmonomorphism[\tilde E]})^{- 1} \canonicalmonomorphism[E] \varphi \, (g s'^1 \varphi) \\
& = ((g \varphi \tilde s^1) (g s'^1 \varphi)^{- 1}) (\canonicalmonomorphism[\tilde E]|^{\Image \canonicalmonomorphism[\tilde E]})^{- 1} \canonicalmonomorphism[\tilde E] \, (g s'^1 \varphi) = (g \varphi \tilde s^1) (g s'^1 \varphi)^{- 1} (g s'^1 \varphi) = g \varphi \tilde s^1
\end{align*}
for \(g \in \Image \structuremorphism[E]\), that is, \(s^1 (\ModulePart \varphi) = (\GroupPart \varphi)|_{\Image \structuremorphism[E]}^{\Image \structuremorphism[\tilde E]} \tilde s^1\).

Conversely, given arbitrary sections \(s^1\) and \(s'^1\) of the underlying pointed map of \(\structuremorphism[E]|^{\Image \structuremorphism[E]}\) such that \(s^1\) fulfills \(s^1 (\ModulePart \varphi) = (\GroupPart \varphi)|_{\Image \structuremorphism[E]}^{\Image \structuremorphism[\tilde E]} \tilde s^1\), it follows that \(g s^1 \structuremorphism[E] = g = g s'^1 \structuremorphism[E]\) for all \(g \in \Image \structuremorphism[E]\), that is, \((g s^1) (g s'^1)^{- 1} \in \Kernel \structuremorphism[E] = \Image \canonicalmonomorphism[E]\) and hence
\begin{align*}
g s^1 & = (g s^1) (g s'^1)^{- 1} (g s'^1) = \big( (g s^1) (g s'^1)^{- 1} \big) (\canonicalmonomorphism[E]|^{\Image \canonicalmonomorphism[E]})^{- 1} \canonicalmonomorphism[E] \, (g s'^1) \\
& = \big( (g s^1) (g s'^1)^{- 1} \big) \varphi (\canonicalmonomorphism[\tilde E]|^{\Image \canonicalmonomorphism[\tilde E]})^{- 1} \canonicalmonomorphism[E] \, (g s'^1) = \big( (g s^1 \varphi) (g s'^1 \varphi)^{- 1} \big) (\canonicalmonomorphism[\tilde E]|^{\Image \canonicalmonomorphism[\tilde E]})^{- 1} \canonicalmonomorphism[E] \, (g s'^1) \\
& = \big( (g \varphi \tilde s^1) (g s'^1 \varphi)^{- 1} \big) (\canonicalmonomorphism[\tilde E]|^{\Image \canonicalmonomorphism[\tilde E]})^{- 1} \canonicalmonomorphism[E] \, (g s'^1)
\end{align*}
for all \(g \in \Image \structuremorphism[E]\). \qedhere
\end{enumerate}
\end{proof}

\begin{corollary} \label{cor:existence_of_compatible_choices_along_crossed_module_extensions}
We let \(E\) and \(\tilde E\) be crossed module extensions of \(G\) with \(M\) and we let \(\varphi\colon E \map \tilde E\) be an extension equivalence.
\begin{enumerate}
\item \label{cor:existence_of_compatible_choices_along_crossed_module_extensions:lifting_systems} There exists a lifting system \((Z^2, Z^1)\) for \(E\) and a lifting system \((\tilde Z^2, \tilde Z^1)\) for \(\tilde E\) with \(\tilde Z^1 = Z^1 (\GroupPart \varphi)\) and \(\tilde Z^2 = Z^2 (\ModulePart \varphi)\).
\item \label{cor:existence_of_compatible_choices_along_crossed_module_extensions:section_systems} There exists a section system \((s^1, s^0)\) for \(E\) and a section system \((\tilde s^1, \tilde s^0)\) for \(\tilde E\) with \(\tilde s^0 = s^0 (\GroupPart \varphi)\) and \(s^1 (\ModulePart \varphi) = (\GroupPart \varphi)|_{\Image \structuremorphism[E]}^{\Image \structuremorphism[\tilde E]} \tilde s^1\).
\end{enumerate}
\end{corollary}
\begin{proof} \
\begin{enumerate}
\item This follows from proposition~\ref{prop:construction_of_compatible_choices_along_crossed_module_extensions}\ref{prop:construction_of_compatible_choices_along_crossed_module_extensions:lifting_systems}.
\item We choose a section \(s^0\) of the underlying pointed map of \(\canonicalepimorphism[E]\) and a section \(\tilde s^1\) of the underlying pointed map of \(\structuremorphism[\tilde E]|^{\Image \structuremorphism[\tilde E]}\). By proposition~\ref{prop:construction_of_compatible_choices_along_crossed_module_extensions}\ref{prop:construction_of_compatible_choices_along_crossed_module_extensions:section_of_the_canonical_epimorphism}, \(\tilde s^0 := s_0 (\GroupPart \varphi)\) is a section of the underlying pointed map of \(\canonicalepimorphism[\tilde E]\), and by proposition~\ref{prop:construction_of_compatible_choices_along_crossed_module_extensions:sections_of_structure_morphism}, there exists a unique section \(s^1\) of the underlying pointed map of \(\structuremorphism[E]|^{\Image \structuremorphism[E]}\) such that \(s^1 (\ModulePart \varphi) = (\GroupPart \varphi)|_{\Image \structuremorphism[E]}^{\Image \structuremorphism[\tilde E]} \tilde s^1\). \qedhere
\end{enumerate}
\end{proof}

We remark that corollary~\ref{cor:existence_of_compatible_choices_along_crossed_module_extensions}\ref{cor:existence_of_compatible_choices_along_crossed_module_extensions:section_systems} will be used in~\cite[prop.~(4.11)]{thomas:2009:on_the_second_cohomology_group_of_a_simplicial_group}.

\begin{corollary} \label{cor:equal_cohomology_classes_for_equivalent_crossed_module_extensions}
The map \(\cohomologyclassofextension\colon \Extensions{2}(G, M) \map \CohomologyGroup[3]_{\text{cpt}}(G, M)\) induces a well-defined map
\[\cohomologyclassofextension\colon \ExtensionClasses{2}(G, M) \map \CohomologyGroup[3]_{\text{cpt}}(G, M), [E]_{\extensionequivalent} \mapsto \cohomologyclassofextension(E).\]
\end{corollary}
\begin{proof}
We suppose given crossed module extensions \(E\) and \(\tilde E\) of \(G\) with \(M\) and an extension equivalence \(\varphi\colon E \map \tilde E\). Then by corollary~\ref{cor:existence_of_compatible_choices_along_crossed_module_extensions}\ref{cor:existence_of_compatible_choices_along_crossed_module_extensions:lifting_systems}, there exists a lifting system \((Z^2, Z^1)\) for \(E\) and a lifting system \((\tilde Z^2, \tilde Z^1)\) for \(\tilde E\) with \(\tilde Z^1 = Z^1 (\GroupPart \varphi)\) and \(\tilde Z^2 = Z^2 (\ModulePart \varphi)\) and proposition~\ref{prop:compatible_choices_along_crossed_module_extensions_provide_compatible_cocycles}\ref{prop:compatible_choices_along_crossed_module_extensions_provide_compatible_cocycles:lifting_systems} implies \(\cohomologyclassofextension(\tilde E) = \cohomologyclassofextension(E)\).
\end{proof}

\begin{definition}[cohomology class associated to a crossed module extension class] \label{def:cohomology_class_associated_to_a_crossed_module_extension_class} \
Given a crossed module extension \(E\) of \(G\) with \(M\), the \(3\)-cohomology class \(\cohomologyclassofextension([E]_{\extensionequivalent}) = \cohomologyclassofextension(E) \in \CohomologyGroup[3]_{\text{cpt}}(G, M)\) is also called the \newnotion{cohomology class associated to \([E]_{\extensionequivalent}\)}.
\end{definition}

\begin{proposition} \label{prop:every_3-cocycle_of_the_cohomology_class_of_an_crossed_module_extension_is_constructable}
We suppose given a crossed module extension \(E\) of \(G\) with \(M\) and a \(3\)-cocycle \(z^3 \in \CocycleGroup[3]_{\text{cpt}}(G, M)\) with \(\cohomologyclassofextension(E) = z^3 + \CoboundaryGroup[3]_{\text{cpt}}(G, M)\). For every lift \(Z^1\) of \(\id_G\) along the underlying pointed map of \(\canonicalepimorphism\), there exists a map \(Z^2\colon G \cart G \map \ModulePart E\) such that \((Z^2, Z^1)\) is a lifting system for \(E\) with \(\cocycleofextension{3}_{E, (Z^2, Z^1)} = z^3\).
\end{proposition}
\begin{proof}
We suppose given a lift \(Z^1\) of \(\id_G\) along the underlying pointed map of \(\canonicalepimorphism\) and we choose an arbitrary componentwise pointed map \(\tilde Z^2\colon G \cart G \map \ModulePart E\) such that \((\tilde Z^2, Z^1)\) is a lifting system for \(E\), that is, with \(\tilde Z^2 (\structuremorphism[E]|^{\Image \structuremorphism[E]}) = \cocycleofextension{2}_{E, Z^1}\). Then we have
\[\cocycleofextension{3}_{E, (\tilde Z^2, Z^1)} + \CoboundaryGroup[3]_{\text{cpt}}(G, M) = \cohomologyclassofextension(E) = z^3 + \CoboundaryGroup[3]_{\text{cpt}}(G, M),\]
that is, there exists a componentwise pointed \(2\)-cochain \(c^2 \in \CochainComplex[2]_{\text{cpt}}(G, M)\) with \(\cocycleofextension{3}_{E, (\tilde Z^2, Z^1)} = (c^2 \differential) + z^3\). Setting \((h, g) Z^2 := ((h, g) c^2 \canonicalmonomorphism[E])^{- 1} (h, g) \tilde Z^2\), proposition~\ref{prop:choices_of_liftings_of_the_non-abelian_2-cocycle_correspond_to_3-cocycles} implies that \((Z^2, Z^1)\) is a lifting system for \(E\) with \(\cocycleofextension{3}_{E, (Z^2, Z^1)} = z^3\).
\end{proof}

We let \(E\) be a crossed module extension of \(G\) with \(M\) and we let \(z^3 \in \CocycleGroup[3]_{\text{cpt}}(G, M)\) be given with \(\cohomologyclassofextension(E) = z^3 + \CoboundaryGroup[3]_{\text{cpt}}(G, M)\). By the preceding proposition, every lift \(Z^1\) of \(\id_G\) along the underlying pointed map of \(\canonicalepimorphism\) can be completed to a lifting system \((Z^2, Z^1)\) for \(E\) such that \(z^3\) is the \(3\)-cocycle of \(E\) with respect to \((Z^2, Z^1)\). A lift of \(\id_G\) along the underlying pointed map of \(\canonicalepimorphism\) is nothing but a section of the underlying pointed map of \(\canonicalepimorphism\). The following example shows that in general not every section \(s^0\) of the underlying pointed map of \(\canonicalepimorphism\) can be completed to a section system \((s^1, s^0)\) for \(E\) such that \(\cocycleofextension{3}\) is the \(3\)-cocycle of \(E\) with respect to \((s^1, s^0)\).

\begin{example} \label{ex:non-trivial_coboundary_is_not_3-cocycle_of_trivial_extension_with_respect_to_a_section_system}
We suppose that \(G\), \(M\) and the action of \(G\) on \(M\) are non-trivial. Then the coboundary group \(\CoboundaryGroup[3]_{\text{cpt}}(G, M)\) has also at least one non-trivial element: We choose \(g_0 \in G \setminus \{1\}\) and \(m_0 \in M \setminus \{1\}\) such that \(g_0 m_0 \neq m_0\), and we define \(c^2 \in \CochainComplex[2]_{\text{cpt}}(G, M)\) by \((h, g) c^2 := \kdelta_{(h, g), (g_0, g_0)} m\) for \(g, h \in G\). This leads to
\[(g_0, g_0, g_0) (c^2 \differential) = (g_0, g_0) c^2 - (g_0, g_0^2) c^2 + (g_0^2, g_0) c^2 - g_0 \act (g_0, g_0) c^2 = m_0 - g_0 m_0 \neq 0.\]
By example~\ref{ex:3-cocycle_of_the_trivial_crossed_module_extension_with_respect_to_the_unique_section_system}, the \(3\)-cocycle of the trivial crossed module extension \(\TrivialHomomorphismCrossedModule{M}{G}\) with respect to the unique section system for \(\TrivialHomomorphismCrossedModule{M}{G}\) is the trivial cocycle \(0 \in \CocycleGroup[3]_{\text{cpt}}(G, M)\). Thus there cannot be a section system for \(\TrivialHomomorphismCrossedModule{M}{G}\) leading to \(c^2 \differential\).
\end{example}

\section{The standard extension} \label{sec:the_standard_extension}

In this section, we construct for a given cocycle \(z^3 \in \CocycleGroup[3]_{\text{cpt}}(G, M)\), where \(G\) is a group and \(M\) is an abelian \(G\)-module, a crossed module extension of \(G\) with \(M\) whose associated cohomology class is \(z^3 + \CoboundaryGroup[3]_{\text{cpt}}(G, M)\). The arguments are adapted from~\cite[sec.~7]{maclane:1949:cohomology_theory_in_abstract_groups_III_operator_homomorphisms_of_kernels},~\cite[sec.~9]{eilenberg_maclane:1947:cohomology_theory_in_abstract_groups_II_group_extensions_with_a_non-abelian_kernel}. At the very end, this leads to the classification theorem~\ref{th:bijection_between_crossed_module_extension_classes_and_third_componentwise_pointed cohomology_group} due to \eigenname{Holt}~\cite[th.~4.5]{holt:1979:an_interpretation_of_the_cohomology_groups_h_n_g_m}, \eigenname{Huebschmann}~\cite[p.~310]{huebschmann:1980:crossed_n-fold_extensions_of_groups_and_cohomology} and \eigenname{Ratcliffe}~\cite[th.~9.4]{ratcliffe:1980:crossed_extensions}.

Throughout the whole section, we suppose given a group \(G\) and an abelian \(G\)-module \(M\).

Before we begin with the construction of the standard extension, we recall two facts from group theory.

\begin{remark} \label{rem:criterion_when_two_group_homomorphisms_with_common_target_induce_a_group_homomorphism_from_the_direct_product}
We suppose given group homomorphisms \(\varphi_1\colon G_1 \map H\) and \(\varphi_2\colon G_2 \map H\). The map \linebreak \(\varphi\colon G_1 \directprod G_2 \map H, (g_1, g_2) \map (g_1 \varphi_1) (g_2 \varphi_2)\) is a group homomorphism if and only if
\[(g_1 \varphi_1) (g_2 \varphi_2) = (g_2 \varphi_2) (g_1 \varphi_1)\]
for all \(g_1 \in G_1\), \(g_2 \in G_2\).
\end{remark}
\begin{proof}
If \(\varphi\) is a group homomorphism, then we necessarily have
\[(g_1 \varphi_1) (g_2 \varphi_2) = (g_1, g_2) \varphi = ((1, g_2) (g_1, 1)) \varphi = (1, g_2) \varphi (g_1, 1) \varphi = (1 \varphi_1) (g_2 \varphi_2) (g_1 \varphi_1) (1 \varphi_2) = (g_2 \varphi_2) (g_1 \varphi_1)\]
for all \(g_1 \in G_1\), \(g_2 \in G_2\). Conversely, if \((g_1 \varphi_1) (g_2 \varphi_2) = (g_2 \varphi_2) (g_1 \varphi_1)\) for all \(g_1 \in G_1\), \(g_2 \in G_2\), then we obtain
\begin{align*}
((g_1, g_2) (g_1', g_2')) \varphi & = (g_1 g_1', g_2 g_2') \varphi = (g_1 g_1') \varphi_1 (g_2 g_2') \varphi_2 = (g_1 \varphi_1) (g_1' \varphi_1) (g_2 \varphi_2) (g_2' \varphi_2) \\
& = (g_1 \varphi_1) (g_2 \varphi_2) (g_1' \varphi_1) (g_2' \varphi_2) = (g_1, g_2) \varphi (g_1', g_2') \varphi
\end{align*}
for all \(g_1, g_1' \in G_1\), \(g_2, g_2' \in G_2\), that is, \(\varphi\) is a group homomorphism.
\end{proof}

\begin{remark} \label{rem:generators_of_the_image_group_of_the_standard_extension}
We let \(F\) be a free group on the underlying pointed set of \(G\) with basis \(s^0\colon G \map F\), and we let \(\pi\) be the extension of \(\id_G\colon G \map G\) to \(F\).
\begin{enumerate}
\item \label{rem:generators_of_the_image_group_of_the_standard_extension:free_group} The kernel \(\Kernel \pi\) is a free group on the pointed set \(\Image \cocycleofextension{2}_{F, s^0}\).
\item \label{rem:generators_of_the_image_group_of_the_standard_extension:point} Given \(g, h \in G\), we have \((h, g) \cocycleofextension{2}_{F, s^0} = 1\) if and only if \(g = 1\) or \(h = 1\).
\item \label{rem:generators_of_the_image_group_of_the_standard_extension:non-points} Given \(g, g', h, h' \in G \setminus \{1\}\) with \((h, g) \cocycleofextension{2}_{F, s^0} = (h', g') \cocycleofextension{2}_{F, s^0}\), it follows that \((h, g) = (h', g')\).
\end{enumerate}
\end{remark}
\begin{proof} \
\begin{enumerate}
\item This follows from the Nielsen-Schreier theorem, see for example~\cite[\S 36, p.~36]{kurosh:1960:the_theory_of_groups} (cf.\ also~\cite[sec.~7, p.~747]{maclane:1949:cohomology_theory_in_abstract_groups_III_operator_homomorphisms_of_kernels}).
\item For \(g, h \in G\), we have \((h, g) \cocycleofextension{2} = 1\) if and only if \((h s^0) (g s^0) = (h g) s^0\). But since \(F\) is freely generated on the underlying pointed set of \(G\) and the length of \((h g) s^0\) for \(g, h \in G\) is less than \(2\), this is equivalent to \(g = 1\) or \(h = 1\).
\item We suppose given \(g, g', h, h' \in G \setminus \{1\}\) with \((h, g) \cocycleofextension{2} = (h', g') \cocycleofextension{2}\). It follows that
\[(g' s^0)^{- 1} (h' s^0)^{- 1} (h s^0) (g s^0) = ((h' g') s^0)^{- 1} (h g) s^0\]
and hence the length of this element is less or equal than \(2\). But since \(g, g', h, h' \neq 1\), this implies that \((h' s^0)^{- 1} (h s^0) = 1\), that is, \(h s^0 = h' s^0\) and therefore \(h = h'\). Thus we have
\[(h g') s^0 (g' s^0)^{- 1} (g s^0) = (h g) s^0\]
and hence the length of \((h g') s^0 (g' s^0)^{- 1} (g s^0)\) is less or equal than \(1\). But now \(g, g' \neq 1\) and \(h g' \neq g'\) implies \(g' = g\). \qedhere
\end{enumerate}
\end{proof}

\begin{proposition} \label{prop:standard_extension_of_a_3-cocycle}
We suppose given a \(3\)-cocycle \(z^3 \in \CocycleGroup[3]_{\text{cpt}}(G, M)\).

We let \(F\) be a free group on the underlying pointed set of \(G\) with basis \(s^0 = Z^1\colon G \map F\). We let \(\pi\) be the extension of \(\id_G\colon G \map G\) to \(F\). The basis \(s^0\) is a section of the underlying pointed map of \(\pi\). We let \(\iota\colon M \map M \directprod \Kernel \pi, m \mapsto (m, 1)\) and \(\mu\colon M \directprod \Kernel \pi \map F, (m, r) \mapsto r\). We let \(s^1\colon \Kernel \pi \map M \directprod \Kernel \pi, r \mapsto (1, r)\) and we let \(Z^2\colon G \cart G \map M \directprod \Kernel \pi\) be given by \(Z^2 := \cocycleofextension{2}_{F, s^0} s^1\). The direct product \(M \directprod \Kernel \pi\) is generated by \(\Image \iota \union \Image Z^2\) and carries the structure of an \(F\)-module uniquely determined on this set of generators by \({^{k Z^1}}(m \iota) := (k m) \iota\) for \(m \in M\), \(k \in G\), and \({^{k Z^1}}((h, g) Z^2) := ((k, h, g) z^3 \iota)^{- 1} ({^{k Z^1}}((h, g) \cocycleofextension{2})) s^1 = ((k, h, g) z^3 \iota)^{- 1} (k, h) Z^2 (k h, g) Z^2 ((k, h g) Z^2)^{- 1}\) for \(g, h, k \in G\).

These data define a crossed module extension \(\StandardExtension(z^3)\) and a section system \((\standardsectionsystem_{z^3}^1, \standardsectionsystem_{z^3}^0)\) for \(\StandardExtension(z^3)\) as follows. The group part of \(\StandardExtension(z^3)\) is given by \(\GroupPart \StandardExtension(z^3) := F\), the module part is given by \(\ModulePart \StandardExtension(z^3) := M \directprod \Kernel \pi\) and the structure morphism is given by \(\structuremorphism[\StandardExtension(z^3)] := \mu\). We have the canonical monomorphism \(\canonicalmonomorphism[\StandardExtension(z^3)] := \iota\) and the canonical epimorphism \(\canonicalepimorphism[\StandardExtension(z^3)] := \pi\). The section system \((\standardsectionsystem_{z^3}^1, \standardsectionsystem_{z^3}^0)\) is defined by \(\standardsectionsystem_{z^3}^0 := s^0\) and \(\standardsectionsystem_{z^3}^1 := s^1\).

By construction, the \(3\)-cocycle of \(\StandardExtension(z^3)\) with respect to the section system \((\standardsectionsystem_{z^3}^1, \standardsectionsystem_{z^3}^0)\) is
\[\cocycleofextension{3}_{\StandardExtension(z^3), (\standardsectionsystem_{z^3}^1, \standardsectionsystem_{z^3}^0)} = z^3.\]
In particular, \(\cohomologyclassofextension(\StandardExtension(z^3)) = z^3 + \CoboundaryGroup[3]_{\text{cpt}}(G, M)\).
\end{proposition}
\begin{proof}
Since the kernel \(\Kernel \pi\) is a free group on the pointed set \(\Image \cocycleofextension{2} = \{(h Z^1) (g Z^1) ((h g) Z^1)^{- 1} \mid g, h \in G\}\) by remark~\ref{rem:generators_of_the_image_group_of_the_standard_extension}\ref{rem:generators_of_the_image_group_of_the_standard_extension:free_group}, the direct product \(M \directprod \Kernel \pi\) is generated by \(\Image \iota \union \Image Z^2\). By definition of \(\iota\), \(\mu\) and \(\pi\), we have an exact sequence of groups
\[M \morphism[\iota] M \directprod \Kernel \pi \morphism[\mu] F \morphism[\pi] G\]
with \(\iota\) injective and \(\pi\) surjective. The pointed map \(s^1\) is a section of the underlying pointed map of \(\mu|^{\Image \mu}\) and hence \(Z^2\) is a lift of \(\cocycleofextension{2}\) along the underlying pointed map of \(\mu|^{\Image \mu}\).

We want to define an action of \(F\) on \(M \directprod \Kernel \pi\), that is, a group homomorphism \(\alpha\colon F^\op \map \AutomorphismGroup(M \directprod \Kernel \pi)\). Since \(F\) is a free group on the underlying pointed set of \(G\), it suffices to define a pointed map \(a\colon G \map \AutomorphismGroup(M \directprod \Kernel \pi)\). So we suppose given a group element \(k \in G\). Since \(M\) is a \(G\)-module, we have a group automorphism \(M \map M\), % manual format
\(m \mapsto k m\) and hence a group homomorphism \(M \map M \directprod \Kernel \pi, m \mapsto (k m) \iota\). Moreover, since \(\Kernel \pi\) is a free group on the pointed set \(\Image \cocycleofextension{2}\), the pointed map \(\Kernel \pi \map M \directprod \Kernel \pi, (h, g) \cocycleofextension{2} \mapsto ((k, h, g) z^3 \iota)^{- 1} ({^{k Z^1}}((h, g) \cocycleofextension{2})) s^1\), which is well-defined by remark~\ref{rem:generators_of_the_image_group_of_the_standard_extension}, extends to a group homomorphism \(\Kernel \pi \map M \directprod \Kernel \pi\). By remark~\ref{rem:criterion_when_two_group_homomorphisms_with_common_target_induce_a_group_homomorphism_from_the_direct_product}, we obtain a group homomorphism \(k a\colon M \directprod \Kernel \pi \map M \directprod \Kernel \pi\), which is given on the elements of \(\Image \iota \union \Image Z^2\) by \(m \iota (k a) = (k m) \iota\) for \(m \in M\) resp.\ by \((h, g) Z^2 (k a) = ((k, h, g) z^3 \iota)^{- 1} ({^{k Z^1}}((h, g) \cocycleofextension{2})) s^1\) for \(g, h \in G\). Now remark~\ref{rem:sections_of_group_epimorphisms_lead_to_non-abelian_2-cocycles} implies that
\begin{align*}
(h, g) Z^2 (k a) & = ((k, h, g) z^3 \iota)^{- 1} ({^{k Z^1}}((h, g) \cocycleofextension{2})) s^1 = ((k, h, g) z^3 \iota)^{- 1} ((k, h) \cocycleofextension{2} (k h, g) \cocycleofextension{2} ((k, h g) \cocycleofextension{2})^{- 1}) s^1 \\
& = ((k, h, g) z^3 \iota)^{- 1} (k, h) Z^2 (k h, g) Z^2 ((k, h g) Z^2)^{- 1}
\end{align*}
for all \(g, h \in G\). To show that \(k a\) is an automorphism on \(M \directprod \Kernel \pi\), we show that the group homomorphism \(k b\colon M \directprod \Kernel \pi \map M \directprod \Kernel \pi\), which is defined on the elements of \(\Image \iota \union \Image Z^2\) by \(m \iota (k b) := (k^{- 1} m) \iota\) for \(m \in M\) resp.\ \((h, g) Z^2 (k b) := ((k^{- 1}, h, g) z^3 \iota)^{- 1} ({^{(k Z^1)^{- 1}}}((h, g) \cocycleofextension{2})) s^1\), is inverse to \(k a\). (The proof that \(k b\) exists and is uniquely determined on \(\Image \iota \union \Image Z^2\) is done in the same way as that of \(k a\).) Indeed, we have
\[m \iota (k a) (k b) = (k m) \iota (k b) = (k^{- 1} k m) \iota = m \iota\]
and
\[m \iota (k b) (k a) = (k^{- 1} m) \iota (k a) = (k k^{- 1} m) \iota = m \iota\]
for all \(m \in M\) as well as
\begin{align*}
& (h, g) Z^2 (k a) (k b) = \big( ((k, h, g) z^3 \iota)^{- 1} (k, h) Z^2 (k h, g) Z^2 ((k, h g) Z^2)^{- 1} \big) (k b) \\
& = \big( (k, h, g) z^3 \iota (k b) \big)^{- 1} \big( (k, h) Z^2 (k b) \big) \big( (k h, g) Z^2 (k b) \big) \big( (k, h g) Z^2 (k b) \big)^{- 1} \\
& = \big( (k^{- 1} (k, h, g) z^3) \iota \big)^{- 1} \big( ((k^{- 1}, k, h) z^3 \iota)^{- 1} ({^{(k Z^1)^{- 1}}}((k, h) \cocycleofextension{2})) s^1 \big) \big( ((k^{- 1}, k h, g) z^3 \iota)^{- 1} ({^{(k Z^1)^{- 1}}}((k h, g) \cocycleofextension{2})) s^1 \big) \\
& \qquad \big( ((k^{- 1}, k, h g) z^3 \iota)^{- 1} ({^{(k Z^1)^{- 1}}}((k, h g) \cocycleofextension{2})) s^1 \big)^{- 1} \\
& = \big( (k^{- 1}, k, h) z^3 \iota \big)^{- 1} \big( (k^{- 1}, k, h g) z^3 \iota \big) \big( (k^{- 1}, k h, g) z^3 \iota \big)^{- 1} \big( (k^{- 1} (k, h, g) z^3) \iota \big)^{- 1} \big( ({^{(k Z^1)^{- 1}}}((k, h) \cocycleofextension{2})) s^1 \big) \\
& \qquad \big( ({^{(k Z^1)^{- 1}}}((k h, g) \cocycleofextension{2})) s^1 \big) \big(({^{(k Z^1)^{- 1}}}((k, h g) \cocycleofextension{2})) s^1 \big)^{- 1} \\
& = \big(- (k^{- 1}, k, h) z^3 + (k^{- 1}, k, h g) z^3 - (k^{- 1}, k h, g) z^3 - k^{- 1} (k, h, g) z^3 \big) \iota \\
& \qquad \big( {^{(k Z^1)^{- 1}}}((k, h) \cocycleofextension{2} (k h, g) \cocycleofextension{2} ((k, h g) \cocycleofextension{2})^{- 1}) \big) s^1 \\
& = ({^{(k Z^1)^{- 1}}}({^{k Z^1}}((h, g) \cocycleofextension{2}))) s^1 = (h, g) \cocycleofextension{2} s^1 = (h, g) Z^2
\end{align*}
and
\begin{align*}
& (h, g) Z^2 (k b) (k a) = \big( ((k^{- 1}, h, g) z^3 \iota)^{- 1} ({^{(k Z^1)^{- 1}}}((h, g) \cocycleofextension{2})) s^1 \big) (k a) \\
& = \big( ((k^{- 1}, h, g) z^3 \iota)^{- 1} ({^{(k Z^1)^{- 1} (k^{- 1} Z^1)^{- 1} (k^{- 1} Z^1)}}((h, g) \cocycleofextension{2})) s^1 \big) (k a) \\
& = \big( ((k^{- 1}, h, g) z^3 \iota)^{- 1} ({^{((k^{- 1}, k) \cocycleofextension{2})^{- 1}}}((k^{- 1}, h) \cocycleofextension{2} (k^{- 1} h, g) \cocycleofextension{2} ((k^{- 1}, h g) \cocycleofextension{2})^{- 1})) s^1 \big) (k a) \\
& = \big( ((k^{- 1}, h, g) z^3 \iota)^{- 1} \, {^{((k^{- 1}, k) Z^2)^{- 1}}}((k^{- 1}, h) Z^2 (k^{- 1} h, g) Z^2 ((k^{- 1}, h g) Z^2)^{- 1}) \big) (k a) \\
& = \big( (k^{- 1}, h, g) z^3 \iota (k a) \big)^{- 1} \, {^{((k^{- 1}, k) Z^2 (k a))^{- 1}}}\big( ((k^{- 1}, h) Z^2 (k a)) ((k^{- 1} h, g) Z^2 (k a)) ((k^{- 1}, h g) Z^2 (k a))^{- 1} \big) \\
& = \big( (k (k^{- 1}, h, g) z^3) \iota \big)^{- 1} \, {^{(((k, k^{- 1}, k) z^3 \iota)^{- 1} ({^{k Z^1}}((k^{- 1}, k) \cocycleofextension{2})) s^1)^{- 1}}}\big( (((k, k^{- 1}, h) z^3 \iota)^{- 1} ({^{k Z^1}}((k^{- 1}, h) \cocycleofextension{2})) s^1) \\
& \qquad (((k, k^{- 1} h, g) z^3 \iota)^{- 1} ({^{k Z^1}}((k^{- 1} h, g) \cocycleofextension{2})) s^1) (((k, k^{- 1}, h g) z^3 \iota)^{- 1} ({^{k Z^1}}((k^{- 1}, h g) \cocycleofextension{2})) s^1)^{- 1} \big) \\
& = \big( (k, k^{- 1}, h) z^3 \iota \big)^{- 1} \big( (k, k^{- 1}, h g) z^3 \iota \big) \big( (k, k^{- 1} h, g) z^3 \iota \big)^{- 1} \big( (k (k^{- 1}, h, g) z^3) \iota \big)^{- 1} \\
& \qquad {^{((k, k^{- 1}) \cocycleofextension{2} s^1)^{- 1}}} \big( ({^{k Z^1}}((k^{- 1}, h) \cocycleofextension{2})) s^1 ({^{k Z^1}}((k^{- 1} h, g) \cocycleofextension{2})) s^1 (({^{k Z^1}}((k^{- 1}, h g) \cocycleofextension{2})) s^1)^{- 1} \big) \\
& = \big( - (k, k^{- 1}, h) z^3 + (k, k^{- 1}, h g) z^3 - (k, k^{- 1} h, g) z^3 - (k (k^{- 1}, h, g) z^3 \big) \iota \\
& \qquad \big( {^{((k, k^{- 1}) \cocycleofextension{2})^{- 1} (k Z^1)}}(((k^{- 1}, h) \cocycleofextension{2}) ((k^{- 1} h, g) \cocycleofextension{2}) ((k^{- 1}, h g) \cocycleofextension{2})^{- 1}) \big) s^1 \\
& = \big( {^{(k^{- 1} Z^1)^{- 1}}}({^{k^{- 1} Z^1}}((h, g) \cocycleofextension{2})) \big) s^1 = (h, g) \cocycleofextension{2} s^1 = (h, g) Z^2
\end{align*}
for all \(g, h \in G\), that is, \((k a) (k b) = (k b) (k a) = \id_{M \directprod \Kernel \pi}\) and hence \(k a \in \AutomorphismGroup(M \directprod \Kernel \pi)\). Altogether, we have defined a map \(a\colon G \map \AutomorphismGroup(M \directprod \Kernel \pi)\). This map is pointed since
\[m \iota (1 a) = (1 m) \iota = m \iota\]
for all \(m \in M\) and
\[(h, g) Z^2 (1 a) = ((1, h, g) z^3 \iota)^{- 1} ({^{1 Z^1}}((h, g) \cocycleofextension{2})) s^1 = (h, g) \cocycleofextension{2} s^1 = (h, g) Z^2\]
for all \(g, h \in G\). Therefore, we get a group homomorphism \(\alpha\colon F^{\op} \map \AutomorphismGroup(M \directprod \Kernel \pi)\) with \(a = Z^1 \alpha\), turning \(M \directprod \Kernel \pi\) into an \(F\)-module with
\[{^{k Z^1}}(m \iota) = m \iota (k \alpha) = m \iota (k a) = (k m) \iota\]
for \(m \in M\), \(k \in G\), and
\begin{align*}
{^{k Z^1}}((h, g) Z^2) & = (h, g) Z^2 (k \alpha) = (h, g) Z^2 (k a) = ((k, h, g) z^3 \iota)^{- 1} ({^{k Z^1}}((h, g) \cocycleofextension{2})) s^1 \\
& = ((k, h, g) z^3 \iota)^{- 1} ((k, h) Z^2) ((k h, g) Z^2) ((k, h g) Z^2)^{- 1}
\end{align*}
for \(g, h, k \in G\). So we have 
\[((k, h) Z^2) ((k h, g) Z^2) = ((k, h, g) z^3 \iota) \, {^{k Z^1}}\!((h, g) Z^2) ((k, h g) Z^2)\]
for \(g, h, k \in G\).

We want to show that the group \(F\), the \(F\)-module \(M \directprod \Kernel \pi\) and the homomorphism \(\mu\colon M \directprod \Kernel \pi \map F\) define a crossed module.
\begin{itemize}
\item[(Equi)] We have
\[({^{k Z^1}}(m \iota)) \mu = (k m) \iota \mu = 1 = {^{k Z^1}}(m \iota \mu)\]
for \(m \in M\) and
\[({^{k Z^1}}((h, g) Z^2)) \mu = \big( ((k, h, g) z^3 \iota)^{- 1} ({^{k Z^1}}((h, g) \cocycleofextension{2})) s^1 \big) \mu = {^{k Z^1}}((h, g) \cocycleofextension{2}) = {^{k Z^1}}((h, g) Z^2 \mu)\]
for \(g, h \in G\).
\item[(Peif)] We have
\[{^{n \iota \mu}}(m \iota) = m \iota = (n + m + (- n)) \iota = {^{n \iota}}(m \iota)\]
for \(m, n \in M\) as well as
\[{^{n \iota \mu}}((h, g) Z^2) = (h, g) Z^2 = {^{n \iota}}((h, g) Z^2)\]
for \(g, h \in G\), \(n \in M\). Moreover, we have
\[{^{(l Z^1) (k Z^1)}}(m \iota) = {^{l Z^1}}({^{k Z^1}}(m \iota)) = {^{l Z^1}}((k m) \iota) = (l (k m)) \iota = ((l k) m) \iota = {^{(l k) Z^1}}(m \iota)\]
for \(m \in M\), \(k, l \in G\), and therefore
\[{^{(l, k) Z^2 \mu}}(m \iota) = {^{(l, k) \cocycleofextension{2}}}(m \iota) = {^{(l Z^1) (k Z^1) ((l k) Z^1)^{- 1}}}(m \iota) = m \iota = {^{(l, k) Z^2}}(m \iota)\]
for \(m \in M\), \(k, l \in G\). Finally, we have
\begin{align*}
& {^{(l Z^1) (k Z^1)}}((h, g) Z^2) = {^{l Z^1}}({^{k Z^1}}((h, g) Z^2)) = {^{l Z^1}}\big( ((k, h, g) z^3 \iota)^{- 1} ((k, h) Z^2) ((k h, g) Z^2) ((k, h g) Z^2)^{- 1} \big) \\
& = {^{l Z^1}}((k, h, g) z^3 \iota)^{- 1} \, {^{l Z^1}}\!((k, h) Z^2) \, {^{l Z^1}}\!((k h, g) Z^2) \, ({^{l Z^1}}\!((k, h g) Z^2))^{- 1} \\
& = \big( (l (k, h, g) z^3) \iota \big)^{- 1} \big( ((l, k, h) z^3 \iota)^{- 1} ({^{l Z^1}}((k, h) \cocycleofextension{2})) s^1 \big) \big( ((l, k h, g) z^3 \iota)^{- 1} ({^{l Z^1}}((k h, g) \cocycleofextension{2})) s^1 \big) \\
& \qquad \big( ((l, k, h g) z^3 \iota)^{- 1} ({^{l Z^1}}((k, h g) \cocycleofextension{2})) s^1 \big)^{- 1} \\
& = ((l (k, h, g) z^3) \iota)^{- 1} ((l, k, h) z^3 \iota)^{- 1} ((l, k h, g) z^3 \iota)^{- 1} (l, k, h g) z^3 \iota \, ({^{l Z^1}}((k, h) \cocycleofextension{2})) s^1 ({^{l Z^1}}((k h, g) \cocycleofextension{2})) s^1 \\
& \qquad (({^{l Z^1}}((k, h g) \cocycleofextension{2})) s^1)^{- 1} \\
& = \big(- (l, k, h) z^3 + (l, k, h g) z^3 - (l, k h, g) z^3 - l (k, h, g) z^3 \big) \iota \, \big( {^{l Z^1}}((k, h) \cocycleofextension{2} (k h, g) \cocycleofextension{2} ((k, h g) \cocycleofextension{2})^{- 1}) \big) s^1 \\
& = (- (l k, h, g) z^3 ) \iota ( {^{l Z^1}}({^{k Z^1}}((h, g) \cocycleofextension{2}))) s^1 = ((l k, h, g) z^3 \iota)^{- 1} \, ({^{(l Z^1) (k Z^1)}}((h, g) \cocycleofextension{2})) s^1 \\
& = ((l k, h, g) z^3 \iota )^{- 1} \, ({^{(l, k) \cocycleofextension{2} (l k) Z^1}}((h, g) \cocycleofextension{2})) s^1 = ((l k, h, g) z^3 \iota )^{- 1} \, {^{(l, k) Z^2}}(({^{(l k) Z^1}}((h, g) \cocycleofextension{2})) s^1) \\
& = {^{(l, k) Z^2}}\big( ((l k, h, g) z^3 \iota )^{- 1} \, ({^{(l k) Z^1}}((h, g) \cocycleofextension{2})) s^1 \big) = {^{(l, k) Z^2}}({^{(l k) Z^1}}((h, g) Z^2))
\end{align*}
for \(g, h, k, l \in G\), and hence we get
\begin{align*}
{^{(l, k) Z^2 \mu}}((h, g) Z^2) & = {^{(l, k) \cocycleofextension{2}}}((h, g) Z^2) = {^{(l Z^1) (k Z^1) ((l k) Z^1)^{- 1}}}((h, g) Z^2) \\
& = {^{(l Z^1) (k Z^1)}}({^{((l k) Z^1)^{- 1}}}((h, g) Z^2)) = {^{(l, k) Z^2}}((h, g) Z^2)
\end{align*}
for \(g, h, k, l \in G\).
\end{itemize}
Altogether, we have constructed a crossed module \(\StandardExtension(z^3)\) with \(\ModulePart \StandardExtension(z^3) = M \directprod \Kernel \pi\), \(\GroupPart \StandardExtension(z^3) = F\) and \(\structuremorphism[\StandardExtension(z^3)] = \mu\). Since the induced action of \(G\) on \(M\) is by definition given by the a priori given action of \(G\) on \(M\), we even have a crossed module extension \(\StandardExtension(z^3)\) of \(G\) with \(M\), where \(\canonicalmonomorphism[\StandardExtension(z^3)] = \iota\) and \(\canonicalepimorphism[\StandardExtension(z^3)] = \pi\).
\end{proof}

\begin{corollary} \label{cor:surjectivity_of_cohomology_class_of_extension}
We suppose given a Grothendieck universe \(\mathfrak{U}\) that contains an infinite set and we suppose \(G\) to be a group in \(\mathfrak{U}\) and \(M\) to be an abelian \(G\)-module in \(\mathfrak{U}\). The map \(\cohomologyclassofextension\colon \Extensions[\mathfrak{U}]{2}(G, M) \map \CohomologyGroup[3]_{\text{cpt}}(G, M)\) is surjective.
\end{corollary}
\begin{proof}
This follows since the free group construction can be done in \(\mathfrak{U}\), provided \(\mathfrak{U}\) contains an infinite set.
\end{proof}

\begin{definition}[standard extension with respect to a \(3\)-cocycle] \label{def:standard_extension_with_respect_to_a_3-cocycle}
We suppose given a \(3\)-cocycle \(z^3 \in \CocycleGroup[3]_{\text{cpt}}(G, M)\). The crossed module extension \(\StandardExtension(z^3)\) as constructed in proposition~\ref{prop:standard_extension_of_a_3-cocycle} is called the \newnotion{standard extension} of \(G\) with \(M\) with respect to \(z^3\). The section system \((\standardsectionsystem_{z^3}^1, \standardsectionsystem_{z^3}^0)\) as also defined in loc.\ cit.\ is said to be the \newnotion{standard section system} for \(\StandardExtension(z^3)\) resp.\ \(z^3\).
\end{definition}

Roughly said, the next proposition states that every \(3\)-cocycle of the standard extension \(\StandardExtension(z^3)\) with respect to a given \(3\)-cocycle \(z^3\) comes from a section system. This is a particular feature of the standard extension, cf.\ example~\ref{ex:non-trivial_coboundary_is_not_3-cocycle_of_trivial_extension_with_respect_to_a_section_system}.

\begin{proposition} \label{prop:every_cocycle_of_the_standard_extension_comes_from_a_section_system}
We suppose given a \(3\)-cocycle \(z^3 \in \CocycleGroup[3]_{\text{cpt}}(G, M)\). For every componentwise pointed \(2\){\nbd}cochain \(c^2 \in \CochainComplex[2]_{\text{cpt}}(G, M)\), the group homomorphism \(\tilde s^1\colon \Image \structuremorphism[\StandardExtension(z^3)] \map \ModulePart \StandardExtension(z^3)\) uniquely determined by \((h, g) \cocycleofextension{2}_{\StandardExtension(z^3), \standardsectionsystem_{z^3}^0} \tilde s^1 := (h, g) c^2 \canonicalmonomorphism[\StandardExtension(z^3)] (h, g) \cocycleofextension{2}_{\StandardExtension(z^3), \standardsectionsystem_{z^3}^0} \standardsectionsystem_{z^3}^1\) for \(g, h \in G\) is a section of the underlying pointed map of \(\structuremorphism[\StandardExtension(z^3)]|^{\Image \structuremorphism[\StandardExtension(z^3)]}\) such that
\[\cocycleofextension{3}_{\StandardExtension(z^3), (\tilde s^1, \standardsectionsystem_{z^3}^0)} = c^2 \differential + z^3.\]
\end{proposition}
\begin{proof}
By remark~\ref{rem:generators_of_the_image_group_of_the_standard_extension}, we have a well-defined pointed map \(\Image \cocycleofextension{2} \map \ModulePart \StandardExtension(z^3), (h, g) \cocycleofextension{2} \mapsto (h, g) c^2 (h, g) \cocycleofextension{2} \standardsectionsystem_{z^3}^1\). Since \(\Image \structuremorphism\) is freely generated by the pointed set \(\Image \cocycleofextension{2}\), it follows that there exists a unique group homomorphism \(\tilde s^1\colon \Image \structuremorphism \map \ModulePart \StandardExtension(z^3)\) with \((h, g) \cocycleofextension{2} \tilde s^1 = (h, g) c^2 \canonicalmonomorphism (h, g) \cocycleofextension{2} \standardsectionsystem_{z^3}^1\) for \(g, h \in G\). This group homomorphism \(\tilde s^1\) is a section of the underlying pointed map of \(\structuremorphism\) and thus \((\tilde s^1, \standardsectionsystem_{z^3}^0)\) is a section system for \(\StandardExtension(z^3)\). We denote the lifting system of \(\StandardExtension(z^3)\) coming from \((\standardsectionsystem_{z^3}^1, \standardsectionsystem_{z^3}^0)\) by \((Z^2, Z^1)\), and the one coming from \((\tilde s^1, \standardsectionsystem_{z^3}^0)\) by \((\tilde Z^2, Z^1)\). Then we have
\[(h, g) \tilde Z^2 = (h, g) \cocycleofextension{2} \tilde s^1 = (h, g) c^2 \canonicalmonomorphism (h, g) \cocycleofextension{2} \standardsectionsystem_{z^3}^1 = (h, g) c^2 \canonicalmonomorphism (h, g) Z^2,\]
and hence
\[\cocycleofextension{3}_{\StandardExtension(z^3), (\tilde s^1, \standardsectionsystem_{z^3}^0)} = c^2 \differential + z^3\]
by proposition~\ref{prop:choices_of_liftings_of_the_non-abelian_2-cocycle_correspond_to_3-cocycles}\ref{prop:choices_of_liftings_of_the_non-abelian_2-cocycle_correspond_to_3-cocycles:3-cocycle_of_lifting}.
\end{proof}

\begin{proposition} \label{prop:extension_equivalence_from_standard_extension_to_arbitrary_crossed_module_extension}
We suppose given a crossed module extension \(E\) of \(G\) with \(M\), and we choose \(z^3 \in \CocycleGroup[3]_{\text{cpt}}(G, M)\) with \(\cohomologyclassofextension(E) = z^3 + \CoboundaryGroup[3]_{\text{cpt}}(G, M)\). Moreover, we suppose given a lifting system \((Z^2, Z^1)\) for \(E\) with \(\cocycleofextension{3}_{E, (Z^2, Z^1)} = z^3\) and we denote by \((\tilde Z^2, \tilde Z^1)\) the lifting system of \(\StandardExtension(z^3)\) coming from the standard section system \((\standardsectionsystem_{z^3}^1, \standardsectionsystem_{z^3}^0)\). There exists a unique extension equivalence \(\omega\colon \StandardExtension(z^3) \map E\) with \(g \tilde Z^1 \omega = g Z^1\) for \(g \in G\) and \((h, g) \tilde Z^2 \omega = (h, g) Z^2\) for \(g, h \in G\).
\[\begin{tikzpicture}[baseline=(m-2-1.base)]
  \matrix (m) [diagram]{
    M & \ModulePart \StandardExtension(z^3) & \GroupPart \StandardExtension(z^3) & G \\
    M & \ModulePart E & \GroupPart E & G \\};
  \path[->, font=\scriptsize]
    (m-1-1) edge node[above] {\(\canonicalmonomorphism[\StandardExtension(z^3)]\)} (m-1-2)
            edge[equality] (m-2-1)
    (m-1-2) edge node[above] {\(\structuremorphism[\StandardExtension(z^3)]\)} (m-1-3)
            edge node[right] {\(\ModulePart \omega\)} (m-2-2)
    (m-1-3) edge node[above] {\(\canonicalepimorphism[\StandardExtension(z^3)]\)} (m-1-4)
            edge node[right] {\(\GroupPart \omega\)} (m-2-3)
    (m-1-4) edge[equality] (m-2-4)
    (m-2-1) edge node[above] {\(\canonicalmonomorphism[E]\)} (m-2-2)
    (m-2-2) edge node[above] {\(\structuremorphism[E]\)} (m-2-3)
    (m-2-3) edge node[above] {\(\canonicalepimorphism[E]\)} (m-2-4);
\end{tikzpicture}\]
In particular,
\[\StandardExtension(z^3) \extensionequivalent E.\]
\end{proposition}
\begin{proof}
Since \(\GroupPart \StandardExtension(z^3)\) is a free group on the underlying pointed set of \(G\), the pointed map \(Z^1\colon G \map \GroupPart E\) extends uniquely to a group homomorphism \(\omega_0\colon \GroupPart \StandardExtension(z^3) \map \GroupPart E\) with \((g \tilde Z^1) \omega_0 = g Z^1\) for \(g \in G\). Further, by remark~\ref{rem:criterion_when_two_group_homomorphisms_with_common_target_induce_a_group_homomorphism_from_the_direct_product} there exists a unique group homomorphism \(\omega_1\colon \ModulePart \StandardExtension(z^3) \map \ModulePart E\) given by \(m \canonicalmonomorphism[\StandardExtension(z^3)] \omega_1 = m \canonicalmonomorphism[E]\) for \(m \in M\) and \(((h, g) \tilde Z^2) \omega_1 = (h, g) Z^2\) for \(g, h \in G\). We get
\[m \canonicalmonomorphism[\StandardExtension(z^3)] \omega_1 \structuremorphism[E] = m \canonicalmonomorphism[E] \structuremorphism[E] = 1 = 1 \omega_0 = m \canonicalmonomorphism[\StandardExtension(z^3)] \structuremorphism[\StandardExtension(z^3)] \omega_0\]
for all \(m \in M\) and
\[(h, g) \tilde Z^2 \omega_1 \structuremorphism[E] = (h, g) Z^2 \structuremorphism[E] = (h, g) \cocycleofextension{2}_{E, Z^1} = (h, g) \cocycleofextension{2}_{\StandardExtension(z^3), \tilde Z^1} \omega_0 = (h, g) \tilde Z^2 \structuremorphism[\StandardExtension(z^3)] \omega_0\]
for all \(g, h \in G\), that is, \(\omega_1 \structuremorphism[E] = \structuremorphism[\StandardExtension(z^3)] \omega_0\). Moreover, we have
\[({^{k \tilde Z^1}}(m \canonicalmonomorphism[\StandardExtension(z^3)])) \omega_1 = (k m) \canonicalmonomorphism[\StandardExtension(z^3)] \omega_1 = (k m) \canonicalmonomorphism[E] = {^{k Z^1}}(m \canonicalmonomorphism[E]) = {^{k \tilde Z^1 \omega_0}}(m \canonicalmonomorphism[\StandardExtension(z^3)] \omega_1)\]
for all \(m \in M\), \(k \in G\), and
\begin{align*}
({^{k \tilde Z^1}}(h, g) \tilde Z^2) \omega_1 & = (((k, h, g) z^3 \canonicalmonomorphism[\StandardExtension(z^3)])^{- 1} (k, h) \tilde Z^2 (k h, g) \tilde Z^2 ((k, g h) \tilde Z^2)^{- 1}) \omega_1 \\
& = ((k, h, g) z^3 \canonicalmonomorphism[\StandardExtension(z^3)] \omega_1)^{- 1} (k, h) \tilde Z^2 \omega_1 (k h, g) \tilde Z^2 \omega_1 ((k, g h) \tilde Z^2 \omega_1)^{- 1} \\
& = ((k, h, g) z^3 \canonicalmonomorphism[E])^{- 1} (k, h) Z^2 (k h, g) Z^2 ((k, g h) Z^2)^{- 1} = {^{k Z^1}}((h, g) Z^2) \\
& = {^{(k \tilde Z^1) \omega_0}}((h, g) \tilde Z^2 \omega_1)
\end{align*}
for all \(g, h, k \in G\), that is, we have a morphism of crossed modules \(\omega\colon \StandardExtension(z^3) \map E\) with \(\ModulePart \omega = \omega_1\) and \(\GroupPart \omega = \omega_0\). Since \(\canonicalmonomorphism[\StandardExtension(z^3)] (\ModulePart \omega) = \canonicalmonomorphism[E]\) and \((\GroupPart \omega) \canonicalepimorphism[E] = \canonicalepimorphism[\StandardExtension(z^3)]\) by construction, we even have an extension equivalence \(\omega\colon \StandardExtension(z^3) \map E\).
\end{proof}

\begin{corollary} \label{cor:extension_equivalence_allows_chains_of_length_2}
We suppose given crossed module extensions \(E\) and \(\tilde E\) of \(G\) with \(M\), and we choose \(z^3 \in \CocycleGroup[3]_{\text{cpt}}(G, M)\) such that \(\cohomologyclassofextension(E) = z^3 + \CoboundaryGroup[3]_{\text{cpt}}(G, M)\). Then \(E \extensionequivalent \tilde E\) if and only if there exist extension equivalences \(\omega\colon \StandardExtension(z^3) \map E\) and \(\tilde \omega\colon \StandardExtension(z^3) \map \tilde E\).
\[\begin{tikzpicture}[baseline=(m-3-1.base)]
  \matrix (m) [diagram]{
    M & \ModulePart E &  \GroupPart E & G \\
    M & \ModulePart \StandardExtension(z^3) & \GroupPart \StandardExtension(z^3) & G \\
    M & \ModulePart \tilde E & \GroupPart \tilde E & G \\};
  \path[->, font=\scriptsize]
    (m-1-1) edge node[above] {\(\canonicalmonomorphism[E]\)} (m-1-2)
    (m-1-2) edge node[above] {\(\structuremorphism[E]\)} (m-1-3)
    (m-1-3) edge node[above] {\(\canonicalepimorphism[E]\)} (m-1-4)
    (m-2-1) edge node[above] {\(\canonicalmonomorphism[\StandardExtension(z^3)]\)} (m-2-2)
            edge[equality] (m-3-1)
            edge[equality] (m-1-1)
    (m-2-2) edge node[above] {\(\structuremorphism[\StandardExtension(z^3)]\)} (m-2-3)
            edge node[right] {\(\ModulePart \tilde \omega\)} (m-3-2)
            edge node[right] {\(\ModulePart \omega\)} (m-1-2)
    (m-2-3) edge node[above] {\(\canonicalepimorphism[\StandardExtension(z^3)]\)} (m-2-4)
            edge node[right] {\(\GroupPart \tilde \omega\)} (m-3-3)
            edge node[right] {\(\GroupPart \omega\)} (m-1-3)
    (m-2-4) edge[equality] (m-3-4)
            edge[equality] (m-1-4)
    (m-3-1) edge node[above] {\(\canonicalmonomorphism[\tilde E]\)} (m-3-2)
    (m-3-2) edge node[above] {\(\structuremorphism[\tilde E]\)} (m-3-3)
    (m-3-3) edge node[above] {\(\canonicalepimorphism[\tilde E]\)} (m-3-4);
\end{tikzpicture}\]
\end{corollary}
\begin{proof}
If there exist extension equivalences \(\omega\colon \StandardExtension(z^3) \map E\) and \(\tilde \omega\colon \StandardExtension(z^3) \map \tilde E\), then \(E \extensionequivalent \tilde E\) by definition. So we suppose conversely that \(E \extensionequivalent \tilde E\). By corollary~\ref{cor:equal_cohomology_classes_for_equivalent_crossed_module_extensions}, this implies \(\cohomologyclassofextension(E) = \cohomologyclassofextension(\tilde E)\). We choose a componentwise pointed \(3\)-cocycle \(z^3 \CocycleGroup[3]_{\text{cpt}}(G, M)\) with \(\cohomologyclassofextension(E) = \cohomologyclassofextension(\tilde E) = z^3 + \CoboundaryGroup[3]_{\text{cpt}}(G, M)\). By proposition~\ref{prop:every_3-cocycle_of_the_cohomology_class_of_an_crossed_module_extension_is_constructable}, there exist lifting systems \((Z^2, Z^1)\) for \(E\) and \((\tilde Z^2, \tilde Z^1)\) for \(\tilde E\) with \(\cocycleofextension{3}_{E, (Z^2, Z^1)} = \cocycleofextension{3}_{\tilde E, (\tilde Z^2, \tilde Z^1)} = z^3\), and we are done with proposition~\ref{prop:extension_equivalence_from_standard_extension_to_arbitrary_crossed_module_extension}.
\end{proof}

As a further corollary, we obtain the fact that the standard extensions with respect to cohomologous cocycles are extension equivalent.

\begin{corollary} \label{cor:standard_extensions_with_respect_to_cohomologous_3-cocycles_are_extension_equivalent}
Given cocycles \(z^3, \tilde z^3 \in \CocycleGroup[3]_{\text{cpt}}(G, M)\) with \(z^3 + \CoboundaryGroup[3]_{\text{cpt}}(G, M) = \tilde z^3 + \CoboundaryGroup[3]_{\text{cpt}}(G, M)\), we have \(\StandardExtension(z^3) \extensionequivalent \StandardExtension(\tilde z^3)\). That is, given a Grothendieck universe \(\mathfrak{U}\) that contains an infinite set and supposed that \(G\) and \(M\) are in \(\mathfrak{U}\), we obtain an induced map
\[\extensionclass\colon \CohomologyGroup[3]_{\text{cpt}}(G, M) \map \ExtensionClasses[\mathfrak{U}]{2}(G, M), z^3 + \CoboundaryGroup[3]_{\text{cpt}}(G, M) \mapsto [\StandardExtension(z^3)]_{\extensionequivalent}.\]
\end{corollary}
\begin{proof}
We suppose given cocycles \(z^3, \tilde z^3 \in \CocycleGroup[3]_{\text{cpt}}(G, M)\) such that \(z^3 + \CoboundaryGroup[3]_{\text{cpt}}(G, M) = \tilde z^3 + \CoboundaryGroup[3]_{\text{cpt}}(G, M)\), and we let \(c^2 \in \CochainComplex[2]_{\text{cpt}}(G, M)\) be given such that \(\tilde z^3 = c^2 \differential + z^3\). By proposition~\ref{prop:every_cocycle_of_the_standard_extension_comes_from_a_section_system}, there exists a section system \((\tilde s^1, \standardsectionsystem_{z^3}^0)\) for \(\StandardExtension(z^3)\) such that \(\tilde z^3\) is the \(3\)-cocycle of \(\StandardExtension(z^3)\) with respect to \((\tilde s^1, \standardsectionsystem_{z^3}^0)\). Hence \(\StandardExtension(\tilde z^3) \extensionequivalent \StandardExtension(z^3)\) by proposition~\ref{prop:extension_equivalence_from_standard_extension_to_arbitrary_crossed_module_extension}.
\end{proof}

\begin{definition}[crossed module extension class associated to a componentwise pointed \(3\)-cohomology class] \label{def:extension_class_associated_to_a_componentwise_pointed_3-cohomology_class}
For every \(3\)-cocycle \(z^3 \in \CocycleGroup[3]_{\text{cpt}}(G, M)\), the crossed module extension class \(\extensionclass(z^3 + \CoboundaryGroup[3]_{\text{cpt}}(G, M)) = [\StandardExtension(z^3)]_{\extensionequivalent}\) is called the \newnotion{crossed module extension class} associated to \(z^3 + \CoboundaryGroup[3]_{\text{cpt}}(G, M)\).
\end{definition}

Finally, we can deduce the desired bijection.

\begin{theorem}[{cf.~\cite[th.~4.5]{holt:1979:an_interpretation_of_the_cohomology_groups_h_n_g_m},~\cite[p.~310]{huebschmann:1980:crossed_n-fold_extensions_of_groups_and_cohomology},~\cite{maclane:1979:historical_note},~\cite[th.~9.4]{ratcliffe:1980:crossed_extensions}}] \label{th:bijection_between_crossed_module_extension_classes_and_third_componentwise_pointed cohomology_group}
We suppose given a Grothendieck universe \(\mathfrak{U}\) that contains an infinite set and we suppose \(G\) to be a group in \(\mathfrak{U}\) and \(M\) to be an abelian \(G\)-module in \(\mathfrak{U}\). The maps
\begin{align*}
& \cohomologyclassofextension\colon \ExtensionClasses[\mathfrak{U}]{2}(G, M) \map \CohomologyGroup[3]_{\text{cpt}}(G, M) \text{ and} \\
& \extensionclass\colon \CohomologyGroup[3]_{\text{cpt}}(G, M) \map \ExtensionClasses[\mathfrak{U}]{2}(G, M)
\end{align*}
are mutually inverse bijections, where the extension class of the trivial crossed module extension \(\TrivialHomomorphismCrossedModule{M}{G}\) corresponds to the trivial cohomology class \(0 \in \CohomologyGroup[3]_{\text{cpt}}(G, M)\). In particular,
\[\ExtensionClasses[\mathfrak{U}]{2}(G, M) \isomorphic \CohomologyGroup[3](G, M).\]
\end{theorem}
\begin{proof}
By proposition~\ref{prop:standard_extension_of_a_3-cocycle}, we have
\[\cohomologyclassofextension(\extensionclass(z^3 + \CoboundaryGroup[3]_{\text{cpt}}(G, M))) = \cohomologyclassofextension([\StandardExtension(z^3)]_{\extensionequivalent}) = \cohomologyclassofextension(\StandardExtension(z^3)) = z^3 + \CoboundaryGroup[3]_{\text{cpt}}(G, M)\]
for every \(3\)-cocycle \(z^3 \in \CocycleGroup[3]_{\text{cpt}}(G, M)\), that is, \(\cohomologyclassofextension \comp \extensionclass = \id_{\CohomologyGroup[3]_{\text{cpt}}(G, M)}\). Thus it remains to show that \(\extensionclass \comp \cohomologyclassofextension = \id_{\ExtensionClasses{2}(G, M)}\). To this end, we suppose given a crossed module extension \(E\) and we choose \(z^3 \in \CocycleGroup[3]_{\text{cpt}}(G, M)\) with \(\cohomologyclassofextension(E) = z^3 + \CoboundaryGroup[3]_{\text{cpt}}(G, M)\). By proposition~\ref{prop:every_3-cocycle_of_the_cohomology_class_of_an_crossed_module_extension_is_constructable}, there exists a lifting system \((Z^2, Z^1)\) for \(E\) such that \(\cocycleofextension{3}_{E, (Z^2, Z^1)} = z^3\). Hence proposition~\ref{prop:extension_equivalence_from_standard_extension_to_arbitrary_crossed_module_extension} implies that \(\StandardExtension(z^3) \extensionequivalent E\) and thus
\[\extensionclass(\cohomologyclassofextension([E]_{\extensionequivalent})) = \extensionclass(\cohomologyclassofextension(E)) = \extensionclass(z^3 + \CoboundaryGroup[3]_{\text{cpt}}(G, M)) = [\StandardExtension(z^3)]_{\extensionequivalent} = [E]_{\extensionequivalent}.\]
The assertion on the trivial crossed module extension has been shown in example~\ref{ex:cohomology_class_of_trivial_extension}. Finally, corollary~\ref{cor:descriptions_of_componentwise_pointed_cocycle_coboundary_and_cohomology_group}\ref{cor:descriptions_of_componentwise_pointed_cocycle_coboundary_and_cohomology_group:cohomology} yields
\[\ExtensionClasses[\mathfrak{U}]{2}(G, M) \isomorphic \CohomologyGroup[3]_{\text{cpt}}(G, M) \isomorphic \CohomologyGroup[3](G, M). \qedhere\]
\end{proof}

The standard extension \(\StandardExtension(z^3)\) with respect to a given \(3\)-cocycle \(z^3 \in \CocycleGroup[3]_{\text{cpt}}(G, M)\) involves free groups. So even if \(G\) and \(M\) are both finite, the module part and the group part of \(\StandardExtension(z^3)\) are both infinite. The question occurs whether there exists a crossed module extension \(E\) with \(\cohomologyclassofextension(E) = z^3 + \CoboundaryGroup[3]_{\text{cpt}}(G, M)\) and with \(\GroupPart E\) and \(\ModulePart E\) finite. Such an extension has been constructed explicitly by \eigenname{Ellis}~\cite[proof for \(c = 2\), p.~502]{ellis:1997:spaces_with_finitely_many_non-trivial_homotopy_groups_all_of_which_are_finite}.

% bibliography

\bigskip

{\raggedleft Sebastian Thomas \\ Lehrstuhl D f{\"u}r Mathematik \\ RWTH Aachen University \\ Templergraben 64 \\ 52062 Aachen \\ Germany \\ sebastian.thomas@math.rwth-aachen.de \\ \url{http://www.math.rwth-aachen.de/~Sebastian.Thomas/} \\}


\begin{thebibliography}{16}
\bibitem{artin_grothendieck_verdier:1972:sga_4_1}
  \eigenname{Artin, Michael}; \eigenname{Grothendieck, Alexander}; \eigenname{Verdier, Jean-Louis}. \newblock
  \booktitle{Th{\'e}orie des topos et cohomologie {\'e}tale des sch{\'e}mas. Tome~1: Th{\'e}orie des topos.} Lecture Notes in Mathematics, vol.~269. \newblock
  Springer-Verlag Berlin-New York, 1972. \newblock
  S{\'e}minaire de G{\'e}om{\'e}trie Alg{\'e}brique du Bois-Marie 1963--1964 (SGA4). \newblock
  With the collaboration of \eigenname{N.~Bourbaki}, \eigenname{P.~Deligne} and \eigenname{B.~Saint-Donat}.
\bibitem{brown:1982:cohomology_of_groups}
  \eigenname{Brown, Kenneth~S.} \newblock
  \booktitle{Cohomology of Groups}. Graduate Texts in Mathematics, vol.~87. \newblock
  Springer-Verlag, New York-Berlin, 1982.
\bibitem{conduche:1984:modules_croises_generalises_de_longeur_2}
  \eigenname{Conduch{\'e}, Daniel}. \newblock
  \booktitle{Modules crois{\'e}s g{\'e}n{\'e}ralis{\'e}s de longeur 2}. \newblock
  Journal of Pure and Applied Algebra \textbf{34}(2--3) (1984), pp.~155--178.
\bibitem{eilenberg_maclane:1947:cohomology_theory_in_abstract_groups_I}
  \eigenname{Eilenberg, Samuel}; \eigenname{Mac\,Lane, Saunders}. \newblock
  \booktitle{Cohomology theory in abstract groups. I}. \newblock
  Annals of Mathematics \textbf{48}(1) (1947), pp.~51--78.
\bibitem{eilenberg_maclane:1947:cohomology_theory_in_abstract_groups_II_group_extensions_with_a_non-abelian_kernel}
  \eigenname{Eilenberg, Samuel}; \eigenname{Mac\,Lane, Saunders}. \newblock
  \booktitle{Cohomology theory in abstract groups. II. Group Extensions with a non-Abelian Kernel}. \newblock
  Annals of Mathematics \textbf{48}(2) (1947), pp.~326--341.
\bibitem{ellis:1997:spaces_with_finitely_many_non-trivial_homotopy_groups_all_of_which_are_finite}
  \eigenname{Ellis, Graham~J.} \newblock
  \booktitle{Spaces with finitely many non-trivial homotopy groups all of which are finite}. \newblock
  Topology \textbf{36}(2) (1997), pp.~501--504.
\bibitem{holt:1979:an_interpretation_of_the_cohomology_groups_h_n_g_m}
  \eigenname{Holt, Derek~F.} \newblock
  \booktitle{An interpretation of the cohomology groups \(H^n(G, M)\)}. \newblock
  Journal of Algebra \textbf{60}(2) (1979), pp.~307--320.
\bibitem{huebschmann:1980:crossed_n-fold_extensions_of_groups_and_cohomology}
  \eigenname{Huebschmann, Johannes}. \newblock
  \booktitle{Crossed \(n\)-fold extensions of groups and cohomology}. \newblock
  Commentarii Mathematici Helvetici \textbf{55}(2) (1980), pp.~302--313.
\bibitem{kurosh:1960:the_theory_of_groups}
  \eigenname{Kurosh, A.~G.} \newblock
  \booktitle{The theory of groups. Vol. II}. Second English edition. \newblock
  Chelsea Publishing Company, New York (NY), 1960.
\bibitem{lang:2002:algebra}
  \eigenname{Lang, Serge}. \newblock
  \booktitle{Algebra}. Graduate Texts in Mathematics, vol.~211, third edition. \newblock
  Springer-Verlag, New York, 2002.
\bibitem{loday:1982:spaces_with_finitely_many_non-trivial_homotopy_groups}
  \eigenname{Loday, Jean-Louis}. \newblock
  \booktitle{Spaces with finitely many non-trivial homotopy groups}. \newblock
  Journal of Pure and Applied Algebra \textbf{24}(2) (1982), pp.~179--202.
\bibitem{maclane:1949:cohomology_theory_in_abstract_groups_III_operator_homomorphisms_of_kernels}
  \eigenname{Mac\,Lane, Saunders}. \newblock
  \booktitle{Cohomology theory in abstract groups. III. Operator Homomorphisms of Kernels}. \newblock
  Annals of Mathematics \textbf{50}(3) (1949), pp.~736--761.
\bibitem{maclane:1979:historical_note}
  \eigenname{Mac\,Lane, Saunders}. \newblock
  \booktitle{Historical Note}. \newblock
  Journal of Algebra \textbf{60}(2) (1979), pp.~319--320. \newblock
  Appendix in~\cite{holt:1979:an_interpretation_of_the_cohomology_groups_h_n_g_m}.
\bibitem{ratcliffe:1980:crossed_extensions}
  \eigenname{Ratcliffe, John~G.} \newblock
  \booktitle{Crossed extensions}. \newblock
  Transactions of the American Mathematical Society \textbf{257}(1) (1980), pp.~73--89.
\bibitem{thomas:2007:co_homology_of_crossed_modules}
  \eigenname{Thomas, Sebastian}. \newblock
  \booktitle{(Co)homology of crossed modules}. \newblock
  Diploma thesis, RWTH Aachen University, 2007. \newblock
  \url{http://www.math.rwth-aachen.de/~Sebastian.Thomas/publications/}
\bibitem{thomas:2009:on_the_second_cohomology_group_of_a_simplicial_group}
  \eigenname{Thomas, Sebastian}. \newblock
  \booktitle{On the second cohomology group of a simplicial group}. \newblock
  Homology, Homotopy and Applications \textbf{12}(2) (2010), pp.~167-210.
\end{thebibliography}
\end{document}